\documentclass{article}

\usepackage[T1]{fontenc}
\usepackage[utf8]{inputenc}
\usepackage{lmodern}
\usepackage[english]{babel}
\usepackage[all,cmtip]{xy}
\usepackage{tkz-fct}
\usepackage{tikz}
\usepackage{bm}
\usepackage{amsmath, amsthm, comment,color,graphicx}
\usepackage{amsfonts}
\usepackage{amssymb}
\theoremstyle{definition} 

 \newtheorem{definition}{Definition}[section]
 \newtheorem{remark}[definition]{Remark}

\theoremstyle{plain}      

 \newtheorem{proposition}[definition]{Proposition}
 \newtheorem{theorem}[definition]{Theorem}
 \newtheorem{corollary}[definition]{Corollary}
 \newtheorem{lemma}[definition]{Lemma}

\newtheorem*{theorem*}{Theorem}

\newcommand{\R}{\mathbb R}
\newcommand{\C}{\mathbb C}
\newcommand{\CP}{\mathbb {CP}^1}
\newcommand{\RP}{\mathbb {RP}^1}
\newcommand{\la}{\langle}
\newcommand{\ra}{\rangle}
\renewcommand{\P}{\mathbb {P}}
\renewcommand{\H}{\mathbb H}
\renewcommand{\S}{\mathbb S}
\newcommand{\bp}{\bm p}

\newcommand{\bmm}{\bm m}
\newcommand{\bw}{\bm w}
\newcommand{\bomega}{\bm \omega}
\renewcommand{\phi}{\varphi}
\newcommand{\bphi}{\bm \phi}
\newcommand{\HdC}{{\mathbf H}_\C^2}
\renewcommand{\Re}{{\rm Re}}
\renewcommand{\Im}{{\rm Im}}

\newcommand{\PGL}{\mathrm{PGL}}

\newcommand{\PU}{\mathrm{PU}}

\author{E. Falbel, A. Guilloux, P. Will}

\title{Hilbert metric, beyond convexity}
\date{}

\begin{document}

\maketitle

\begin{abstract}
The Hilbert metric on convex subsets of $\R^n$ has proven a rich notion and has been extensively studied.
We propose here a generalization of this metric to subset of complex projective spaces and give examples
of applications to diverse fields.   Basic examples include the classical Hilbert metric which coincides with the hyperbolic metric on real  hyperbolic spaces as well as the complex hyperbolic metric on complex  hyperbolic spaces.
\end{abstract}

\section{Introduction}

The Hilbert metric on convex subsets of $\R^n$ is a well-known and well-studied
object. We refer the reader to \cite{HandbookHilbert} for a comprehensive introduction
to the metric aspects. A major domain of applications is the field of
divisible convex, where the invariance of the metric under projective transformations
of the convex is leveraged. Such a study originates in large part from
the series of works of Benoist begining with \cite{Benoist}. We refer to the
survey \cite{Benoist-survey} for a more precise description of these works.

We propose in this paper a generalization of the notion of Hilbert
metric to settings without convexity. The main idea to 
overcome the lack of convexity is to make use  of the duality between projective
spaces and dual projective spaces. A first attempt to give such
a generalization has been done by the second author in \cite{Guilloux-padic}. Whereas
the focus of the cited paper was projective spaces over local fields, 
we mainly work here in real and complex
projective spaces. 

We define in Section \ref{sec:definition} a notion of generalized Hilbert metric 
on a set $\Omega$ in $\P(\C^{n+1})$ as long
as $\Omega$ avoids a compact set of hyperplanes. For example, any open set in
$\CP$ inherits such a metric. The metric does not, in general, separate points
but we are able to determine the conditions under which it does (Theorem \ref{thm:metric}).
For example, any open subset of $\CP$ whose complement does not lie in a circle
inherits an actual metric. We then move on to a description of the 
associated infinitesimal Finsler metric
(Theorem \ref{theorem:metric}). An example of such a generalized Hilbert metric is the 
usual hyperbolic metric on 
complex hyperbolic space, just as the hyperbolic metric on the real hyperbolic
space is known to be an example of Hilbert metric. Note that our metric is naturally
invariant under projective transformations of $\Omega$.  A more general problem is to compare the Bergman metric to ours 
for bounded domains in $\C^n$ (see remark \ref{shilov}).

We then give three different directions of application to these definitions. We hope
that the given definition will prove useful in a wealth of problems and try to
convince the reader so. The first direction (Section \ref{sec:complexkleinian}) 
is at the origin of this work: we explore the
meaning of the definition for complex Kleinian hyperbolic groups, i.e. discrete subgroups
of $\PU(n,1)$. We are able to reinterpret in a geometric way
our definition and prove that it defines a natural metric on uniformized spherical CR manifolds. 

The second direction (Section \ref{section:puncturedsphere}) deals with a very 
simple example, akin to the polygon case of the Hilbert metric (see \cite{delaHarpe}): 
the $n$-punctured projective line, where the punctures do not lie in a single circle.
We prove that our metric is quasi-isometric - but not isometric - to the 
hyperbolic metric on the $n$-punctured sphere. 

Eventually, we look in 
Section \ref{section:RP1} at an example that may seem strange at first 
glance: any open subset of the real projective line, even if not connected, inherits a metric. 
We focus on complements of self-similar compact sets (such as limit sets of Fuchsian groups or 
self-similar Cantor sets) and are able to prove a generalization of a formula due to 
Basmajian in the setting of hyperbolic surfaces.

We thank Gilles Courtois and Pascal Dingoyan for several useful discussions. The third author 
would like to thank Hugo Parlier for an enlightening discussion. Many discussions around this project 
occured at the UMPA, and we thank that institution for its kind hospitality.

\section{The metric}\label{sec:definition}

We define here a metric on subsets of projective spaces avoiding a compact set of hyperplanes, under a
non degeneracy hypothesis. The idea is to redefine the usual Hilbert metric in a more suitable way
to a generalization to projective spaces on other fields than $\R$. A first attempt, with mainly the $p$-adic fields in mind, has been done in \cite{Guilloux-padic}. We are interested in this paper in the complex case, so we will always assume that we are working in complex projective spaces. Note that the real case follows, by including the real projective spaces in its complexification.

We will regularly switch between elements in projective spaces and lifts in the vector spaces. In order to be able to do it without cumbersome definitions and notations, we fix a convention throughout this paper:
points in projective spaces will be denoted by symbols like $\omega$, $\phi$, $m$, $p$... and any lift will be denoted by the bold corresponding symbol $\bomega$, $\bphi$, $\bmm$, $\bp$...

\subsection{Setting and first examples}

We fix throughout the following section an integer $n\geq 1$. We denote by $\P$ the $n$-dimensional projective space $\P(\C^{n+1})$ and by $\P'$ its dual
$\P((\C^{n+1})^\vee)$. We consider two non-empty subsets $\Omega \subset \P$ and 
$\Lambda \subset \P'$ such that

\begin{equation}\label{eq:condi-subset}\forall \omega \in \Omega,\, \forall \phi \in \Lambda \ \ \ \ \bphi(\bomega)\neq 0.\end{equation}

Geometrically, each point in $\P'$ represents a hyperplane in $\P$, and condition 
\eqref{eq:condi-subset} means that
$\Omega$ is disjoint from all hyperplanes defined by points in $\Lambda$. 
The following definition gives a name to such pairs.
\begin{definition}
A pair $(\Omega,\Lambda)$, where $\Omega$ is open in $\P$, $\Lambda$ closed in $\P'$ and
condition \ref{eq:condi-subset} holds, is called \emph{admissible}.
 
Morever, we say that $\Omega$ is \emph{saturated} if it is the set of points on which forms
of $\Lambda$ do not vanish.
\end{definition}

We will define a weak metric on the set $\Omega$ of admissible pairs. Here are the examples of 
typical situations we are 
thinking of. As hinted by the notation, such examples often come with considering $\Lambda$ to be a
limit set in $\P'$ for the action of a subgroup $\Gamma \subset \PGL(n+1,\C)$ and
$\Omega$ to be the set of points on which elements of $\Lambda$ do not vanish.
\begin{enumerate}
\item If $\Omega$ is a bounded domain of $\C^n\subset \P$, we can take
$\Lambda$ to be the set of complex lines that do not meet $\Omega$. The pseudo-metric
we will define on these examples may not  be complete or even separate points. But in
the case of the unit ball, we recover the Bergman metric (see Remark \ref{rem:complexhyperbolicspace}).
\item Let $\Gamma$ be a quasi-Fuchsian subgroup of $\PGL(2,\C)$, 
with limit set $\Lambda_\Gamma\subset \CP$. Note that for $n=1$,
then $\CP = \P$ is naturally identified to its dual $\P'$: the
isomorphism sends a point in $\CP$ to its orthogonal i.e. the class
of forms vanishing at this point. 
So $\Lambda_\Gamma$ is seen as a subset, our $\Lambda$, of $\P'$.
We thus set $\Omega=\Omega_\Gamma$, the complement in 
$\CP$ of $\Lambda_\Gamma$. In this case, 
the pair $\Omega,\Lambda$ satisfies \eqref{eq:condi-subset}, it is an admissible pair and $\Omega$ is saturated.
 
\item One may consider a discrete subgroup $\Gamma$ of $\PU(n,1)$ and suppose
that its limit set $\Lambda_\Gamma$ in the sphere $\S^{2n-1}$ at 
infinity of complex hyperbolic space $\H^n_\C$ is not the whole sphere.
For each point $p\in \S^{2n-1}$, there exists a unique projective complex hyperplane 
tangent to $\S^{2n-1}$ at $p$, which we denote by $L_p$. To each point in $\Lambda_\Gamma$ we associate the (class of) linear form
$\phi_p$ defined by $\P(\ker(\bphi_p))=L_p$. 
Note that $\phi_p$ is actually just the class of the bracket $\la ,\bp\ra$, where 
$\bp$ is a lift of $p$ to $\C^{n+1}$, and $\la \cdot,\cdot \ra$ is the ambient Hermitian form of signature $(n,1)$. We then define 
$$\Lambda=\lbrace \phi_p, p\in\Lambda_\Gamma\rbrace,\mbox{ and } \Omega=\underset{p\in\Lambda_\Gamma}{\bigcap} L_p^c,$$
where $L_p^c$ denotes the complement. These two sets satisfy \eqref{eq:condi-subset}. 
We will go into more details in Section \ref{sec:complexkleinian}. We will see 
that this $\Omega$ is the complement of the Kulkarni limit set and interpret our metric in a geometric manner.

\item If $\Omega$ is an open proper convex set in $\R^n \subset \P(\R^{n+1})$, then one take for 
$\Lambda$ the set of forms that do not vanish on $\Omega$. In standard terminology used for convex sets in affine
real space, $\Lambda$ is the closure of the dual convex to $\Omega$. In this case, we have the usual notion of Hilbert metric \cite{delaHarpe,HandbookHilbert} with a huge literature studying various properties of this metric.
This example is our benchmark: everything we define in a more general setting has to specialize to
the usual Hilbert metric for these pairs $(\Omega,\Lambda)$.
\end{enumerate}

\subsection{The cross-ratio}

Our main source of inspiration is the Hilbert metric, whose definition relies on the
notion of cross-ratio of four points on a projective line. We fix in this section the convention
on cross-ratios and
define a natural notion of cross-ratio between two forms and two points and gather
some well-known properties.

We take as a definition of the cross-ratio of four distinct points in $\CP$ with coordinates $(a,b,c,d)$ 
the element $t=[a,b,c,d]$ such that there is an element of $\PGL(2,\C)$ sending the 
4-tuples $(a,b,c,d)$ to 
$(\infty,0,1,t)$. In any affine chart, it is given by:
\begin{equation}\label{eq:usual-crossratio}[a,b,c,d]=\dfrac{(d-b)(c-a)}{(d-a)(c-b)}.\end{equation}

We now define the cross-ratio $[\phi,\phi',\omega,\omega']$ for $\phi$ and $\phi'$ in $\P'$ and 
$\omega$, $\omega'$ in $\P$. For any two $\omega\neq\omega'$ in $\P$, we denote by $(\omega\omega')$ the (complex) projective line they span.
\begin{definition} \label{def:crossratio}
\begin{enumerate}
  \item  For any $\phi\in \Lambda$ and any pair $(\omega,\omega')$ of distinct points in $\Omega$, let 
$\phi_{\omega,\omega'}$ be the point $\ker(\phi)\cap (\omega,\omega')$ in  $(\omega,\omega')$.
\item We call {\it cross-ratio of $(\phi,\phi',\omega,\omega')$} the cross-ratio 
\begin{equation}\label{eq:defi-crossratio}
[\phi,\phi',\omega,\omega'] := [\phi_{\omega,\omega'},\phi'_{\omega,\omega'},\omega,\omega']
\end{equation}
\end{enumerate}
\end{definition}

Note that the four points involved in \eqref{eq:defi-crossratio} lie in a common complex line by definition.
The cross-ratio of $(\phi,\phi',\omega,\omega')$  can be computed simply with lifts $(\bphi,\bphi',
\bomega,\bomega')$ as follows (compare with \cite[Lemma 3.1]{Guilloux-padic}).

\begin{lemma}
 Let $(\phi,\phi',\omega,\omega')\in\Lambda^2\times\Omega^2$. The cross-ratio
$[\phi,\phi',\omega,\omega']$ satisfies
\begin{equation}
\label{eq:comput-crossratio}[\phi,\phi',\omega,\omega']=\dfrac{\bphi(\bomega)\bphi'(\bomega')}{\bphi(\bomega')\bphi'(\bomega)}
\end{equation}
\end{lemma}

\begin{figure}
\centering
\scalebox{0.4}{\includegraphics{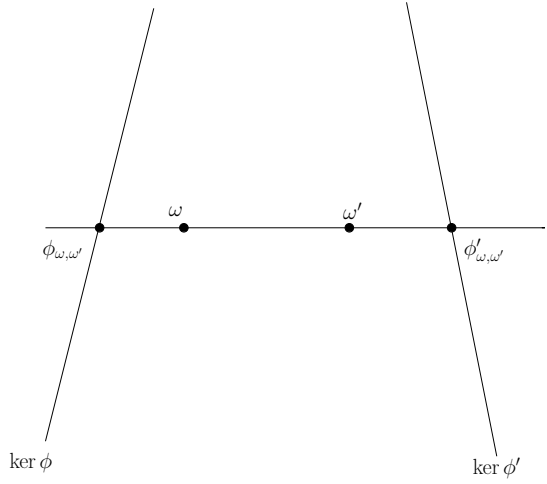}}
\caption{The cross-ratio $[\phi,\phi',\omega,\omega']$}\label{fig:cross-ratio}
\end{figure}

\begin{proof}
The proof is classical and elementary. It is summarized by Figure \ref{fig:cross-ratio}. We include it for completeness.

If $\phi=\phi'$ the identity is obvious.  If $\phi\neq \phi'$,
choose a basis $(e_k)_{1\leqslant k\leqslant n+1}$ of $\C^{n+1}$ and lifts
 such that $\bphi$, $\bphi'$, $\bphi_{\omega,\omega'}$ and $\bphi'_{\omega,\omega'}$ 
are as follows:
$$\bphi = e_1^\vee, \bphi'= e_2^\vee, \bphi_{\omega,\omega'}= e_2+ \bw, 
\bphi'_{\omega,\omega'} = e_1+\bw',$$
where $\bw$ and $\bw'$ are vectors in ${\rm Span}(e_3\cdots e_{n+1})$.
Then $\bomega$ and $\bomega'$ have the form
$$\bomega = \lambda \bphi_{\omega,\omega'} + \lambda'\bphi'_{\omega,\omega'}\mbox{ and } \bomega'= \mu \bphi_{\omega,\omega'} + \mu'\bphi'_{\omega,\omega'}.$$
Computing the right hand side of \eqref{eq:comput-crossratio}, we obtain
\begin{equation*} 
\dfrac{\bphi(\bomega)\bphi'(\bomega')}{\bphi(\bomega')\bphi'(\bomega)}  = 
\dfrac{\mu\lambda'}{\lambda \mu'}  = 
[\phi_{\omega,\omega'},\phi'_{\omega,\omega'},\omega,\omega'],
\end{equation*}
where the second equality is obtain using \eqref{eq:usual-crossratio} by noting 
that with the chosen coordinates, 
the four points $\bphi_{\omega,\omega'}$, $\bphi'_{\omega,\omega'}$, 
$\bomega$ and $\bomega'$ are given by
$$ \bphi_{\omega,\omega'}\sim 0,\, \bphi'_{\omega,\omega'}
\sim \infty,\, \bomega\sim\dfrac{\lambda}{\lambda'}
\mbox{ and } \bomega'\sim\dfrac{\mu}{\mu'}.$$
\end{proof}

The following identities follow by a direct verification.
\begin{proposition}\label{prop:birapport-identities}
Let $\omega,\omega',\omega''$ be three points in $\Omega$, and $\phi,\phi',\phi''$ be three points in $\Lambda$.
Then
\begin{enumerate}
  \item $[\phi,\phi',\omega,\omega'] = [\phi',\phi,\omega',\omega]$
  \item $[\phi,\phi',\omega,\omega']= [\phi,\phi',\omega,\omega''][\phi,\phi',\omega'',\omega']$
  \item $[\phi,\phi',\omega,\omega']=[\phi,\phi'',\omega,\omega'][\phi'',\phi',\omega,\omega']$
\end{enumerate}
\end{proposition}
The last two equalities are known as \emph{cocycle relation} for the cross-ratio. With these properties at hand, we may proceed to the definition of our metric.

\subsection{A generalized Hilbert pseudo-metric}

From now on, we assume that the set $\Lambda$ is compact. Our Hilbert metric
is defined by the following:
\begin{definition}
Let $d_\Lambda$ be the function defined $\Omega \times \Omega$ by
\begin{equation}\label{eq:def-dist}
d_\Lambda (\omega,\omega') = \ln\left(\max \left\{ 
            \left|\left[\phi,\phi',\omega,\omega'\right]\right| \textrm{ for }
            \phi,\phi'\textrm{ in } \Lambda\right\}\right)
\end{equation}
\end{definition}

As noted in \cite[Section 3.2]{Guilloux-padic}, in the case of open proper convex 
subset of $\P(\R^{n+1})$, 
we recover the usual Hilbert metric up to a factor $\frac{1}{2}$. This formula
is reminiscent of a metric associated to a Funk metric \cite{PapadopoulosTroyanov}.
We will take advantage of this remark later on, by separating the contributions
of $\phi$ and $\phi'$.

In our more general setting, $d_\Lambda$ is not quite a metric but almost:
\begin{proposition}
The function $d_\Lambda$ is a pseudo-metric: it is non-negative, 
symmetric and satisfies to the triangle inequality.
\end{proposition}
In the terms of \cite{PapadopoulosTroyanov-weakMinkowski}, $d_\Lambda$ is a symmetric weak metric.

\begin{proof}
Note that exchanging $\phi$ and $\phi'$ in $\left|\left[\phi,\phi',\omega,\omega'\right]\right|$ transforms it into 
its inverse. This implies in particular that $d_\Lambda$ is non-negative. The other two properties follow directly 
from the first two items of Proposition \ref{prop:birapport-identities}
\end{proof}

In general this metric does not separate points. Indeed, if $\Lambda$ consists of a single point, then every cross-ratio is $1$ and $d_\Lambda$ is trivial. A more interesting example is the case of $\Lambda = \{0,\infty\} \subset \CP$ and $\Omega = \C\setminus\{0\}$. Then, it is easy to compute that, for two non zero complex numbers $z$ and $z'$, we have:
$$d_{\{0,\infty\}} (z,z') = |\ln |z| - \ln |z'|\,|.$$
Thus, points of same modulus are at distance $0$. We will focus our attention to punctured spheres in section \ref{section:puncturedsphere}.

Before exploring the conditions for $d_{\Lambda}$ to be an actual metric, let us point out two 
consequences of the mere definition of this function. First, we remark that $d_\Lambda$ is invariant under projective transformations:
\begin{proposition}\label{prop:invariance}
Let $(\Omega, \Lambda)$ be an admissible pair. For any $g\in \PGL(n+1,\C)$, the pair $(g\cdot \Omega, g\cdot \Lambda)$ is admissible, and the action of $g$ is an isometry between $(\Omega, d_\Lambda)$ and $(g\cdot \Omega, d_{g\cdot \Lambda})$
\end{proposition}

\begin{proof}
Indeed, the cross-ratio defined in Definition \ref{def:crossratio} is invariant under projective transformation.
\end{proof}
As a consequence, when $\Lambda$ is a limit set for a group $\Gamma$ and $\Omega$ its 
complement, as in the first examples described, $d_\Lambda$ is $\Gamma$-invariant.

The second fact we want to point out states the pseudo-convexity of $\Omega$.
\begin{proposition}
Let $(\Omega,\Lambda)$ be an admissible pair, with $\Omega$ satured. 

Then $\Omega$ is pseudo-convex and for any $\omega_0 \in \Omega$, the function 
$\omega \to d_\Lambda(\omega_0,\omega)$ is a subharmonic exhaustion.
\end{proposition}

\begin{proof}
Fix a point $\omega_0$ in $\Omega$, and consider the function 
$F(\omega) = d_{\Lambda}(\omega_0,\omega)$.
This function is defined as the max on a compact set of functions 
$\ln(| [\phi,\phi',\omega_0,\omega]|)$. The cross-ratios are 
holomorphic functions of $\omega \in \Omega$. Hence the function 
$F$ is sub-harmonic. Moreover, if $\omega$ escapes any compact in 
$\Omega$, then there are forms $\phi$ in $\Lambda$ such that $\phi(\omega) \to 0$. 
Then, fix a form $\phi_0 \in \Lambda$: the 
cross-ratio $[\phi,\phi_0,\omega_0,\omega]$ goes to $\infty$ 
hence $F(\omega) \to +\infty$. Conclusion: $F$ is a subharmonic exhaustion 
of $\Omega$.
\end{proof}

\subsection{Separation condition}\label{subsec:separation}

We now explore when two points $\omega$ and $\omega'$ are separated 
by the metric $d_\Lambda$, meaning that $d_\Lambda (\omega,\omega')>0$.
Once $\omega$ and $\omega'$ are fixed, the cross-ratios $[\phi,\phi',\omega,\omega']$
are determined by the points $\phi_{\omega,\omega'}$ as in Definition \ref{def:crossratio}.
Let us define a notation for the set of these points:

\begin{definition}\label{def:projection}
 The set $\Lambda_{\omega,\omega'}$  of points $\phi_{\omega,\omega'}$
    for $\phi \in \Lambda$ is called the \emph{projection} of $\Lambda$ on 
    the line $(\omega,\omega')$.
\end{definition}

As stated in the following proposition, $d_\Lambda$ separates $\omega$ and $\omega'$
as soon as its projection in the complex line $(\omega,\omega')$ is not included in
a real line with $\omega$ and $\omega'$ complex conjugate w.r.t. this line. 

\begin{theorem}\label{thm:metric}
Let  $\omega,\omega'$ be two distinct points in $\Omega$. 
The following three conditions are equivalent.
\begin{enumerate}
 \item $d_\Lambda$ does not separate $\omega$ and $\omega'$.
 \item For all pairs $(\phi,\phi')$ in $\Lambda\times\Lambda$, $\|[\phi,\phi',\omega,\omega']\|=1$.
 \item There exists an anti-holomorphic involution of the complex line 
 $(\omega,\omega')$ which exchange $\omega$ 
and $\omega'$, and fixes pointwise the projection $\Lambda{\omega,\omega'}$.
 \end{enumerate}
Moreover, if $d_\Lambda$ separates each pair of distinct points, then $d_\Lambda$ is a metric.
\end{theorem}
For an admissible pair $(\Omega,\Lambda)$, if $d_\Lambda$ is a metric, we will say that $\Lambda$ 
is \emph{separating}.

\begin{proof}
The first two items of the equivalence are clearly equivalent. Choosing a coordinate on the 
line $(\omega,\omega')$ such that 
$\omega=0$ and $\omega'=\infty$, we see that the condition 
$\|[\phi,\phi',\omega,\omega']\|=1$ is equivalent to
the fact that the two points $\phi_{\omega,\omega'}$ and 
$\phi'_{\omega,\omega'}$ lie on a same circle centered at 
$0$. So if every cross-ratio has modulus one, the whole projection 
$\Lambda_{\omega,\omega'}$ is include in this circle. 
We may assume, up to a change of coordinate, that this circle
is the unit circle. 
The reflection $z \to \frac{1}{\bar z}$ about this circle is an anti-holomorphic 
involution which fixes pointwise $\Lambda_{\omega,\omega'}$, and exchanges 
$\omega = 0$ and $\omega' = \infty$. 

Conversely, suppose that  an anti-holomorphic involution $\sigma$ fixes the projection
$\Lambda_{\omega,\omega'}$ and $\sigma(\omega) = \omega'$. Then, 
up to a change of coordinates, $\sigma$ is the
complex conjugation, $\Lambda_{\omega,\omega'}$ is included in the real line
$\RP$ and $\omega' = \bar \omega$. Then, every cross-ratio has modulus one:
\[ [\phi_{\omega,\omega'},\phi'_{\omega,\omega'},\omega,\omega'] = 
\frac{(\phi_{\omega,\omega'} - \omega)(\phi'_{\omega,\omega'}-\omega')}{\phi_{\omega,\omega'} - \omega')(\phi'_{\omega,\omega'}-\omega)} = 
\frac{\overline{(\phi_{\omega,\omega'} - \omega')(\phi'_{\omega,\omega'}-\omega)}}{(\phi_{\omega,\omega'} - \omega')(\phi'_{\omega,\omega'}-\omega)} \]
 
 This proves the equivalence. The last sentence is straightforward: the separation was the
 only property lacking to $d_\Lambda$ to be a metric.
\end{proof}

\begin{remark}\label{remark:zariski}We shall give examples where this condition holds. In fact, it is not as hard to check as it may seem.
Indeed, the equivalence implies that if $\Lambda$
fails to separate two points, then it is included in a $\R$-Zariski closed subset of $\P'$. Indeed, in this case, the equation $\|[\phi,\phi',\omega,\omega']\|=1$ should be valid for every $\phi$ for any fixed $\phi'$.  The intersection of all solutions for varying $\phi'$ is a $\R-$Zariski closed subset of $\P'$.

In
$\CP$, any subset $\Lambda$ which is not included in a circle defines an actual metric on its
complement $\Omega$.  So this metric on a $3$-punctured sphere is never separating.
Still, we may often 
decide whether $\Lambda$ separates or not, even in higher dimensions, and especially in the case of limit sets. 
In Section \ref{section:CR}, we will interpret this separation condition geometrically 
in the context of spherical CR structures. We will then give non trivial examples of separating 
sets $\Lambda$. 
\end{remark}

\begin{remark}
As we have already noted, if $\Omega$ is an open proper convex subset of $\P(\R^{n+1})$, then
we take $\Lambda$ to be the closure of its dual convex set. In this case, $d_\Lambda$ is the usual
Hilbert metric, up to a factor $\frac 12$. An intriguing remark is that we may take a smaller $\Lambda$,
and define a metric on a disconnected set $\Omega$. For example, in the plane, take 
$\Lambda = \{\phi_1,\phi_2,\phi_3\}$ to be  three forms (not in the same line). Then each component of the complement of the three lines in the projective plane is a triangle. On each triangle,
$d_\Lambda$ is the usual Hilbert metric (see \cite{delaHarpe} for a beautiful study of these), but
then the distance between two points in different triangles is also defined. We will come
back to this remark in section \ref{section:RP1} in the seemingly trivial case of the real line $\RP$.
\end{remark}

\subsection{An infinitesimal symmetric Finsler metric}\label{Finsler metric}

We want to explore the infinitesimal behavior of the metric $d_\Lambda$. Once again,
the real case is the source of inspiration, where it is known that the Hilbert metric is
a Finsler geometry \cite{Troyanov} and this Finsler geometry is an object of study. Beware
that in this Hilbert geometry setting, the notion of Finsler metric is not as smooth
as in other parts of the literature: a Finsler metric on $\Omega$, for our purpose, is
a continuous function on $T\Omega$ which is a norm in each tangent space $T_\omega\Omega$.

In our case the situation is a bit more intricate than in the real case. 
We show in this section that $d_\Lambda$ does indeed define a Finsler metric, and are able to 
compute it. But, the presence of the $\max$ in the definition of $d_\Lambda$ and
the lack of regularity of our $\Lambda$ in examples we consider interesting prevent
any smoothness. Moreover, $(\Omega,d_\Lambda)$ is not in general a length space:
the metric $d_\Lambda$ is not the infimum length of a smooth path between two points. 
We will nonetheless be able to understand when the unit ball for the norms are strictly convex.

\newcommand{\norm}[2]{\|#2\|_{\Lambda,#1}}

We begin by computing the infinitesimal behavior of $d_\Lambda$ to define the Finsler metric. We have to parametrize a tangent space $T_\omega\Omega$. To do that, we choose a lift $\bomega$ and identify 
$T_\omega\Omega$ with a subspace $T\in \C^{n+1}$
transverse to the line $\omega$.
\begin{theorem}\label{theorem:metric}
Let $(\Omega,\Lambda)$ be an admissible pair, with $\Lambda$ separating. 
Then the metric $d_\Lambda$ yields a symmetric Finsler metric $(\omega,v) \to \norm{\omega}{v}$ on $T\Omega$, which is given for $T_\omega\Omega$ and $v\in T$ by:
\begin{equation}\label{formula:Finsler}
 \norm{\omega}{v} = \max_{\phi, \phi' \in \Lambda} 
\: \Re\left(\frac{\bphi'(v)}{\bphi'(\bomega)} - \frac{\bphi(v)}{\bphi(\bomega)}\right)
\end{equation}
\end{theorem}
Note that this formula is actually independent of the choice of lifts $\bphi$, $\bphi'$. Moreover,
multiplying the lift $\bomega$ by some $r \in \C$ amounts to changing the parametrization of 
$T_\omega\Omega$ by $T$. Hence this formula is indeed defined on $T_\omega\Omega$. Another
point worth noting is that formula actually separates the contribution of $\phi$ and $\phi'$,
as in the usual way to pass from a Funk metric to a Finsler metric \cite{PapadopoulosTroyanov}.

\begin{proof}
Let $\omega$ be a point in $\Omega$, with $\bomega \in \C^{n+1}$ a lift. We identify 
$T_{\omega} \Omega$ with a hyperplane 
$T \subset \C^{n+1}$ transverse to the line generated by $m$. 
For $v\in T$ and $t>0$, we will prove that the first
order term in the Taylor expansion of $d_\Lambda(\omega,\omega+tv)$ as 
$t\rightarrow 0$ defines a norm on the tangent space 
at $\omega$. Let us first fix $\phi$ and $\phi'$ two points 
in  $\Lambda$ and $\bphi$ and $\bphi'$ corresponding linear forms. 
We first expand the cross-ratio at first order:
\begin{eqnarray*}
[\phi,\phi',\omega,\omega+tv] = & 
        \dfrac{\bphi(\bomega)\bphi'(\bomega+tv)}{\bphi'(\bomega)\bphi(\bomega+tv)}\\
      = & \dfrac{1+t\frac{\bphi'(v)}{\bphi'(\bomega)}}{1+t\frac{\bphi(v)}{\bphi(\bomega)}}\\
      = & 1 + t \left(\dfrac{\bphi'(v)}{\bphi'(\bomega)} - \dfrac{\bphi(v)}{\bphi(\bomega)}\right) + o(t)
\end{eqnarray*}

Since $\ln(|1+tz|) = \frac 12\log(|1+tz|^2)=\frac 12 \log(1 + t(z+\bar z)+o(t)) = t \Re(z) +o(t)$, we may 
further compute:
$$
\frac{d}{dt}_{|t=0+} \ln\left(\left|[\phi,\phi',\omega,\omega+tv]\right|\right) = 
\Re\left(\frac{\bphi'(v)}{\bphi'(\bomega)} - \frac{\bphi(v)}{\bphi(\bomega)}\right).
$$

The metric is given by $d_{\Lambda} (\omega,\omega+tv)= 
\max_{\phi, \phi' \in \Lambda} \ln\left(\left|[\phi,\phi',\omega,\omega+tv]\right|\right)$. Every 
function appearing in the $\max$ equals $0$ for $t=0$. Lemma \ref{lem:derive} below tells us that we may swap the max 
and the derivative. This gives us the announced expression for $\| \cdot \|_\Lambda$:

\begin{eqnarray}
\norm{\omega}{v} & = & \frac{d}{dt}_{|t=0+} d_\Lambda(\omega,\omega+tv)\nonumber\\
&  = & \max_{\phi, \phi' \in \Lambda} \Re\left(\frac{\bphi'(v)}{\bphi'(\bomega)} - \frac{\bphi(v)}{\bphi(\omega)}\right).\nonumber
\end{eqnarray}
Note that since we are taking the maximum over all pairs $(\phi,\phi')$, 
the quantity is positive: exchanging $\phi$ and $\phi'$ just changes the sign.

To prove that $\| \cdot \|_\Lambda$ defines a symmetric Finsler metric, we need to show that, in each tangent space, the sublevel set 
$B_\Lambda^{\omega}=\{ v\in T_{\omega}\Omega,\norm{\omega}{v}\leqslant 1\}$ is compact, convex and symmetric. This sublevel set may be written as the intersection

\begin{equation}\label{eq:unit-ball}
\underset{\phi,\phi'\in\Lambda}{\bigcap}\Bigl\{ \Re\left(\frac{\bphi'(v)}{\bphi'(\bomega)} - \frac{\bphi(v)}{\bphi(\bomega)}\right)\leqslant 1\Bigr\}.
\end{equation}
We see from \eqref{eq:unit-ball} that $B_\Lambda^{\omega}$ is convex as an intersection of half-spaces, and symmetric : if the max in \eqref{formula:Finsler} for a vector $v$ is obtained for a pair $(\phi,\phi')$, then $(\phi',\phi)$ realizes the max for $-v$.

The last point to verify is the compactness. Closedness follows from \eqref{eq:unit-ball}. 
For any $t>0$, $\Lambda$ is separates the points $\omega$ and $\omega+tv$. 
Therefore, the module of the cross-ratio 
$[\phi,\phi',\omega,\omega+tv]$ is not identically $1$ for any $\phi$, $\phi'$ in $\Lambda$. This implies the existence of 
$\phi$ et $\phi'$ such that $\Re\left(\frac{\bphi'(v)}{\bphi'(\bomega)} - \frac{\bphi(v)}{\bphi(\bomega)}\right)\neq 0$. In other 
terms, for any $v$ we have $\norm{\omega}{v}>0$. 

We conclude by contradiction: suppose $B_\Lambda^{\omega}$ is not bounded. We would have a sequence $v_n$ of tangent 
vectors at $\omega$ such that $|v_n|\rightarrow \infty$ (where $|\cdot|$ is any norm on 
$T_\omega\Omega$), and
$$\forall (\phi,\phi')\in \Lambda\times\Lambda,\,
\Re\left(\frac{\bphi'(v_n)}{\bphi'(\bomega)} - \frac{\bphi(v_n)}{\bphi(\bomega)}\right)\leqslant 1.$$
In particular, up to extraction, the sequence $v_n/|v_n|$ converges to a vector $v$ that satisfies 
$|v|=1$ and 
$\Re\left(\frac{\bphi'(v)}{\bphi'(\bomega)} - \frac{\bphi(v)}{\bphi(\bomega)}\right) = 0$, which is a contradiction.
\end{proof}

We now state the technical lemma used in the proof.
\begin{lemma}\label{lem:derive}
Let $\mathcal F$ be a bounded set of $C^2$-functions from $\R_+$ to $\R$ vanishing at $0$.
Let $f$ be defined by $f(t) = \max_{g\in\mathcal F} g(t)$ for $t\geq 0$. 
Then $f$ has a derivative at $0$ given by $f'(0) = \max_{g\in\mathcal F} g'(0)$.
\end{lemma}

\begin{proof}

Observe first that because $t>0$, 
$$\frac{f(t)}{t} = \frac{\max_{g\in\mathcal F} g(t)}{t}=\max_{g\in\mathcal F} \frac{g(t)}{t}.$$
Consider next a function $g$ in $\mathcal{F}$. The second order expansion of $g$ gives 

$$\dfrac{g(t)}{t}=g'(0)+t\Bigl(\dfrac{g''(0)}{2}+\varepsilon_g(t)\Bigr),$$
where  $\varepsilon_g$ is a continuous function depending on $g$, such that 
$\varepsilon_g(t)\underset{t\rightarrow 0^+}\longrightarrow 0$. The boundedness of $\mathcal{F}$ implies that the two 
sets $\{|\varepsilon_g(t)|,t\in[0,1],g\in\mathcal{F}\}$ and $\{|g''(0)|,g\in\mathcal{F}\}$ are bounded. In turn, there 
exists a constant $C>$ such that

$$
\forall g\in\mathcal{F},\, \forall t\in[0,1],\,
g'(0) -C t \leq \frac{g(t)}{t} \leq g'(0) + Ct.
$$
We obtain therefore 
$$
\max_{\mathcal F}g'(0) -C t \leq \max_{\mathcal F} \frac{g(t)}{t} \leq \max_{\mathcal F} g'(0) + Ct.
$$
This implies $\max_{\mathcal F} \frac{g(t)}{t} \to \max_{\mathcal F} g'(0)$.
\end{proof}

We give here a condition for the unit balls of the Finsler metric to be strictly convex. Recall 
that in the real case, it amounts to the condition that the boundary of $\Omega$ contains no segment.
The condition we give here is the same, translated by duality.

As before, for $\omega \in \Omega$ we choose a lift $\bomega$ and identify $T_\omega\Omega$ to $T\subset \C^{n+1}$. We denote by 
$$\Psi = \left\{ \frac{\bphi'}{\bphi'(\bomega)} - \frac{\bphi}{\bphi(\bomega)}\quad \textrm{where} \quad \phi,\phi' \in \Lambda \right\}.$$

From Theorem \ref{theorem:metric} we know that for all $v\in T$, 
$$\norm{\omega}{v} = \underset{\psi\in \Psi}{\max} \: \Re \left(\psi(v)\right).$$

\begin{proposition}\label{prop:Minkovski}
Let $(\Omega,\Lambda)$ be an admissible pair, with $\Lambda$ separating and $\omega \in \Omega$. 
The unit ball of the Finsler metric $\norm{\omega}{\cdot}$ is strictly convex if and only if 
for any pair of distinct tangent vectors $u$ and $v$ at $\omega$, the $\max$ 
defining $\norm{\omega}{u}$ and $\norm{\omega}{v}$ are not obtained for the same form $\psi \in \Psi$.
\end{proposition}

\begin{proof}

For simplicity, we denote by $f(v)=\norm{\omega}{v} = \max_{\psi\in \Psi} \Re(\psi(v))$.
Note that $f$ is positively homogeneous : $f(\lambda v)=\lambda f(v)$ for all $\lambda \in\R_+$ 
and $v\in T$.

Suppose now that for each  $\psi \in\Psi$, at least one of the terms $\Re(\psi(u))$ and $\Re(\psi(v))$ is not the maximum over all $\psi$'s in $\Psi$. Then 
\begin{eqnarray}
\Re\left(\psi\left(\frac{u+v}{2}\right)\right) & = & \frac{\Re(\psi(u)) + \Re(\psi(v))}{2}\nonumber\\ 
 & < & \frac{\max_{\psi\in\Psi}(\Re(\psi(u))) + \max_{\psi\in\Psi}(\Re(\psi(v)))}{2} \nonumber\\
& =  &\frac{f(u)+f(v)}{2}.\label{ineq-convex} 
\end{eqnarray}
 Since $\Psi$ is compact, we obtain by taking the maximum the strict inequality
$$f\left(\frac{u+v}{2}\right)  < \frac{f(u)+f(v)}{2},$$
which amounts to the strict convexity of balls.
\end{proof}

\begin{remark}
Examples of separating $\Lambda$ for which the balls are not strictly convex are easily built
using this proposition: any finite subset $\Lambda \in \CP$ not included in a circle separates points
in its complement. But from the previous proposition, balls are not strictly convex: indeed, they are
polygons.
\end{remark}

\begin{remark}\label{rem:complexhyperbolicspace}
Let $\C^{n,1}$ be  $\C^{n+1}$ equipped with the Hermitian form
$$
\langle Z, W \rangle=Z_0\overline{W}_{n}+Z_2\overline{W}_2+\cdots Z_{n-1}\overline{W}_{n-1}+Z_{n}\overline{W}_0.
$$
One has three subspaces:
$$
  V_{+}=  \{ Z\in \C^{n,1} : \langle Z, Z \rangle >0 \},
$$
$$
   V_{0}=  \{ Z\in \C^{n,1}- \{0\} : \langle Z, Z \rangle =0 \},
$$
$$
V_{-}=  \{ Z\in \C^{n,1} : \langle Z, Z \rangle < 0 \}.
$$
Let $P : \C^{n,1}- \{0\} \rightarrow \P(\C^{n+1}) $
be the canonical projection onto complex projective space.
Then  complex hyperbolic n-space is defined as $\H^n_{\C}=P(V_{-})$ equipped with the Bergman metric.
The boundary of complex hyperbolic space is defined as $\partial \H^n_{\C}=P(V_{0})$.
Using the hermitian form, we identify $\partial \H^n_{\C}$ as a subset of the dual $\P((\C^{n+1})^\vee)$.

Then the pair $(\H^n_{\C},\partial \H^n_{\C})$ is admissible and separates points. Moreover, from Proposition \ref{prop:invariance}, the Hilbert metric $d_{\partial \H^n_{\C}}$ is invariant under the action of unitary group $\mathrm{U}(n,1)$ of the Hermitian form. Hence this metric is a multiple of the Bergman metric (one may compute that the multiplication factor is $\frac 12$). We will come back to complex hyperbolic geometry in the following section.
\end{remark}

\begin{remark}\label{shilov}
More generally, let $\Omega$ be a domain in $\C^n$.  There are several situations where a natural metric can be associated to it.  A very general 
definition is that of the Bergman metric on any bounded domain.  The construction is such that the biholomorphisms group is contained in the isometry group of that metric.
The particular case where $\Omega$ is a bounded homogeneous domain has been studied for a long time.  It contains the important class of non-compact hermitian symmetric spaces.  These Riemannian spaces, classified by Cartan, can be embedded as 
bounded domains which contain the origin, are stable under the circle action  and which turn out to be convex (see \cite{wolf} for a thorough exposition).   

The group of biholomorphisms of a bounded symmetric domain is transitive and can be extended to the boundary but its action on the boundary is not transitive except in the case of the complex ball.  On the other hand the isotropy group at the origin acts
by linear maps of $\C^n$ and, moreover, it acts transitively on the Shilov boundary of $\Omega$.
Consider the set $\Lambda$ of all hyperplanes tangent to the boundary which touch it in at least two points.  They all pass by the Shilov boundary
and therefore the action of the isotropy preserves $\Lambda$.
The distance $d_\Lambda(0,x)$ is therefore invariant under the isotropy group and, as the action of the isotropy is irreducible, 
$d_\Lambda(0,x)$ coincides up to a scalar with the Bergman metric.  One can then translate this distance the whole domain using the action of the automorphism group.

 As an example,
 consider the bidisc $\Delta\times \Delta \subset \C^2$.   Its Shilov boundary is $S^1\times S^1$ and  the relevant hyperplanes 
 passing through it are of the form $z=z_0\in S^1$ or $w=w_0\in S^1$ where $(z_0,w_0)\in S^1\times S^1$ and $(z,w)$ are coordinates of $\C^2$.  We recuperate then the Bergman metric $\max \{ d_h(x_1,x_2), d_h(y_1,y_2)\}$ for any two points
$(x_1,y_1),(x_2,y_2)\in \Delta\times \Delta$.

\end{remark}

The expository and general discussion of this paper is over. We now begin to explore three different directions, illustrating the wealth of possible applications of the definition of this metric. 
We first reinterpret geometrically the definitions in the situation of a 
discrete subgroup of $\PU(n,1)$. We begin by 
recalling the definition of the Kulkarni limit set and focusing to the particular case of spherical CR 
geometry. We remark that the construction gives a natural metric on uniformized spherical CR structures 
on manifolds.
We proceed in the last two sections with the study of two simple cases in one 
dimension projetive spaces:  
first the complex projective line and the punctured spheres. We then move on to the real projective line, 
with attention to self-similar sets $\Lambda$. We 
give a generalization of the famous Basmajian formula in hyperbolic geometry.
The three following sections are independent.

\section{Complex hyperbolic groups}\label{sec:complexkleinian}

Natural examples of the generalized Hilbert metric can be defined for 
open subsets in $\P(\C^{n+1})$ which are 
domains of discontinuity of discrete subgroups $\Gamma\subset \PGL(n+1,\C)$. 
The general theory of such sets of discontinuity is not yet fully
developed but a special case has been studied which is of major
interest for us: if $\Gamma$ is a complex hyperbolic group, i.e.
a discrete subgroup of $\PU(n,1)$.

We show in this section that if $\Gamma$ is a non-elementary discrete
subgroup of $\PU(n,1)$, then its domain of discontinuity in $\P(\C^{n+1})$
inherits a generalized Hilbert metric. We then focus on the case $n=2$,
and give geometric interpretations of this metric. Recall from Remark 
\ref{rem:complexhyperbolicspace} the notations for the complex
hyperbolic space $\H^n_\C \subset \P(\C^{n+1})$ and
its boundary $\partial\H^n_\C$ as well as the projective
unitary group $\PU(n,1)$.

\subsection{Kulkarni limit set as a union of hyperplanes}
\label{subsec:kulkarni}

In order to define an appropriate set $\Lambda$ of 1-forms we start by recalling 
the definition of limit set for these groups. 
The following definition transposes in this context
a definition due to Kulkarni for general actions of 
groups on topological spaces. As in the general situation,
we denote by $\P$ the projective space $\P(\C^{n+1})$ and
by $\P'$ its dual.

\begin{definition}  Let $\Gamma\subset \PGL(n+1,\C)$ be a discrete subgroup
\begin{enumerate}
\item $L_0(\Gamma)$ is the closure of the set of points in  $\P$ with infinite isotropy.
\item $L_1(\Gamma)$ is the closure of the set of cluster points of the $\Gamma$-orbits of all $z\in \P\setminus L_0(\Gamma)$.
\item $L_2(\Gamma)$ is the closure of the set of cluster points of the $\Gamma$-orbits of all compact subsets $K\subset  \P\setminus (L_0(\Gamma)\cup L_1(\Gamma))$.
\item The Kulkarni limit set $L(\Gamma)$ is the set $L_0(\Gamma)\cup L_1(\Gamma)\cup L_2(\Gamma)$.
\item The Kulkarni discontinuity region $\Omega_\Gamma$ is $\P\setminus L(\Gamma)$.
\item We denote by $\Lambda_\Gamma \subset \P'$ the set of forms whose kernel is included in $L(\Gamma)$.
\end{enumerate}
\end{definition}

In the case of $\Gamma$ a complex hyperbolic group, Cano Liu and Lopez 
\cite[Theorem 0.1]{CanoLiuLopez} prove that $L(\Gamma)$ is the union of kernels
of $\Lambda_\Gamma$ and that $\Omega_\Gamma$ is the largest set on which $\Gamma$
acts properly and discontinuously. Note that one may
even take the a priori smaller $\Lambda_\Gamma$ consisting of form whose kernel is tangent to the sphere at infinity $\partial \H^n_\C$. In view of our definitions, this translate to:

\begin{theorem}\label{thm:complexhyperbolic}
Let $\Gamma$ be a complex hyperbolic group. 
Then the pair $(\Omega_\Gamma,\Lambda_\Gamma)$ is admissible. 
Moreover, if $\Gamma$ is Zariski-dense in $\PU(n,1)$, then $\Lambda_\Gamma$ is separating.
\end{theorem}
\begin{proof}
The fact that the pair $(\Omega_\Gamma,\Lambda_\Gamma)$ is admissible follows from definitions and  \cite[Theorem 0.1]{CanoLiuLopez}.  The separation property follows from the fact that, by contraction, if the metric does not separate two points,
than $\Lambda_\Gamma$ should be contained in a Zariski closed subset of $\P'$ (see remark \ref{remark:zariski}).  But, in that case, the group $\Gamma$ (which 
preserves $\Lambda_\Gamma$) would not be Zariski dense.
\end{proof}

A detailed study of such examples is done in \cite{CanoParkerSeade}, for complex hyperbolic subgroups
that are included in $\mathrm{SO}(n,1)$. They prove that in this case $\Omega_\Gamma$ consists
of $3$ connected components. Their groups are not Zariski-dense in $\mathrm{SU}(n,1)$. 
The metric $d_{\Lambda_\Gamma}$ gives a non-separating metric in such cases.

The metric defined here seems most interesting when restricted to 
$\Omega = \Omega_\Gamma \cap \partial \H^n_\C$, as we explain in the following section.

\subsection{Complex hyperbolic geometry and spherical CR geometry}\label{section:CR}

We give here a geometric reinterpretation of the separability condition (Theorem \ref{thm:metric}) 
in the case of complex hyperbolic groups acting on $\P(\C^{n+1})$. We then give examples of discrete 
complex hyperbolic groups for which we can check this separability condition. Those groups arise
from the construction of spherical CR geometric structure on $3$-manifolds that are uniformizable. One
example is a representation of the $8$-knot complement group in $\PU(2,1)$ associated to a uniformisable 
spherical CR structure on the $8$-knot complement, constructed in \cite{DerauxFalbel}. Another one is
 constructed for the Whitehead link complement in \cite{ParkerWill}.

We begin by recalling 
the definition if some geometric objects in complex hyperbolic geometry, mainly the bisectors.
A {\sl bisector} is the locus of  points equidistant  from two given points $p_0$ and $p_1$ 
in $\H^n_\C$.  In homogeneous coordinates, 
  it is given by the
negative vectors $\mathbf{z}=(z_0,z_1,z_2)$ that satisfy the equation
\begin{equation*}\label{eq:bisector}
|\langle \mathbf{z},\tilde{p}_0\rangle|=|\langle
\mathbf{z},\tilde{p}_1\rangle|,
\end{equation*}
where $\tilde{p}_i$ are lifts satisfying $\langle \tilde{p}_0,\tilde{p}_0 \rangle=\langle
\tilde{p}_1,\tilde{p}_1 \rangle$.  This equation makes sense up to the boundary $\partial \H^n_\C$ defining {\sl spinal spheres} as boundaries of bisectors.
Observe that bisectors and spinal spheres are defined by algebraic equations.
CR structures appear naturally as boundaries of complex manifolds.  The local geometry
of these structures was studied by E. Cartan \cite{C} who defined, in dimension three, a curvature analogous to
curvatures of a Riemannian structure.  When that curvature is zero,
Cartan called them spherical CR structures and developed their basic properties.  A much later  study by
Burns and Shnider \cite{BS} contains the modern setting for these structures.

\begin{definition}
A spherical CR-structure on a $(2n-1)$-dimensional manifold is a geometric structure
modeled on the homogeneous space $\S^{2n-1}:=\partial \H^{2n-1}_{\C}$ with the above
${\PU(n,1)}$ action.
\end{definition}

A particular class of such structures is the natural analog of complete structure for
metric geometries:
\begin{definition}
We say a spherical CR-structure on a ${2n-1}$-manifold is uniformizable if it is
equivalent to a quotient of the domain of discontinuity in $ \S^{2n-1}$ of a
discrete subgroup of ${\PU(2,1)}$.
\end{definition}

Here, equivalence between CR structures is defined, as usual, by
diffeomorphisms preserving the structure. A $(2n-1)$-manifold $M$ with a spherical CR structure 
is said to be an 
\emph{uniformized CR spherical manifold} if there is a discrete group $\Gamma$ of $\PU(n,1)$,
with limit set $\Lambda$ in the sphere $\partial \H^n_\C$ such that the spherical CR structure on $M$
is equivalent to the quotient of $\Omega = \partial \H^n_\C\setminus \Lambda$ by $\Gamma$.

\subsection{Invariant metric for uniformized CR spherical manifold}

Let $M$ be a uniformized CR spherical manifold. Denote as above by $\Gamma$ the discrete group
with limit set $\Lambda$ and domain of discontinuity $\Omega$ in the sphere $\partial \H^n_\C$
such that $M \simeq \Gamma \backslash \Omega$.

We identify, as in section \ref{subsec:kulkarni}, $\Lambda$ with the subset of the dual projective space
$\P'$ consisting of forms whose kernel is tangent to the sphere at points of $\Lambda$. Then, Theorem
\ref{thm:complexhyperbolic} states that $d_\Lambda$ defines a $\Gamma$-invariant, non-separating 
metric on the domain of
discontinuity of $\Gamma$ in $\P(\C^{n+1})$. We look here at this metric restricted to the hypersurface
$\Omega$. We reintrepret here the definition \ref{eq:def-dist} in terms more classical
in the framework of complex hyperbolic geometry. We the proceed with a geometric interpretation of the
separation condition given in Section \ref{subsec:separation}.

Let $\phi$ and $\phi'$ be 1-forms whose kernel are, respectively, complex tangent spaces at points $p$ and $p'$ in $\Lambda\in \S^{2n-1}$. Choose lifts $\bp$ and $\bp'$ of $p$ and $p'$.
Then lifts of $\phi$ and $\phi'$ are given explicitly by $\bphi(z)=\langle z,\bp\rangle$ and $\bphi'(z)=\langle z,\bp'\rangle$.  One computes, for $\omega,\omega'\in \Omega$,

$$[\phi,\phi',\omega,\omega'] =\dfrac{\langle \bomega,\bp\rangle\langle\bomega',\bp'\rangle}{\langle \bomega,\bp'\rangle\langle\bomega',\bp\rangle}$$
which is the hermitian cross-ratio of the four points $p,p',\omega,\omega'$.  

Let $\omega,\omega'$ be two points in $ \S^{2n-1}$.  If $p\in \S^{2n-1}$ is another point then its projection in the geodesic in the complex disc defined by
$\omega$ and $\omega'$ is 
$$\pi_{(\omega\omega')}(p)=t \bomega + \dfrac{1}{t}\bomega'\mbox{ with } t=\sqrt{\dfrac{\lvert\langle { \bp},{ \bomega} \rangle\rvert}{\lvert\langle {  \bp},{ \bomega} \rangle\vert}},$$

As a direct corollary, we obtain a condition for cross-ratios to have modulus $1$.
\begin{proposition}
Let $(a,b,c,d)$ be four pairwise distinct points  in $ \S^{2n-1}$.
The cross-ratio $[a,b,c,d]$ has modulus $1$ if and only the projection of $c$ and $d$
on the geodesic $(ab)$ coincide. 
\end{proposition}

One can compare with the condition in Theorem \ref{thm:metric} for the set $\Lambda$ to separate 
a pair $\omega,\omega'$.  Indeed the above proposition shows that $\Lambda$ separates this
pair if its projection on the complex disc defined by
$\omega$ and $\omega$ is contained in a geodesic orthogonal to the geodesic defined by $\omega,\omega'$.
Observe also that we can exchange the roles of $a,b$ and $c,d$ in the previous proposition.

\begin{corollary}
If $\Lambda$ is not contained in a bisector then $d_{\Lambda}$ is a metric.
\end{corollary}

\begin{proof}
 The inverse image of a geodesic by a projection in a complex disc is a bisector.  By the proposition, 
 $\Lambda$ does not separate two points $\omega,\omega'$ if 
$\Lambda$ is included in a bisector.
\end{proof}

We may rewrite the discussion following Theorem \ref{thm:metric}. One can think of two cases where 
$d_\Lambda$ will not separate points in $\Omega$: suppose $\Gamma$ is included in a subgroup
$\PU(n-1,1)$ or $\mathrm{PO}(n,1)$, or in other terms that it preserves a totally geodesic submanifold. 
Then its limit set would be included in a bisector, which would prevent the 
separation condition. The following theorem states that it is essentially the only problem.  Recall that in Theorem \ref{thm:complexhyperbolic} we proved that the metric on the complement of the Kulkarni limit set is separating.
The following theorem restates this result considering only points in the boundary of complex hyperbolic space in order to stress that 
the metric provides a distance on the regular set of the action of $\Gamma$ on $\S^{2n-1}$.

\begin{theorem}\label{thm:metriccomplexhyp}
Let $\Gamma$ be a discrete subgroup of $\PU(n,1)$  
with limit set a proper subset of  $\S^{2n-1}$. 
Suppose that $\Gamma$ is Zariski dense in $\PU(n,1)$.
Then $d_{\Lambda}$ defines a $\Gamma$-invariant metric on the regular set 
$\Omega$.  
\end{theorem}

A practical criterium to verify that the subgroup $\Gamma$ is Zariski dense and that the distance is well defined is given in the following corollary. 

\begin{corollary}
Let $\Gamma$ be a discrete subgroup of $\PU(n,1)$  
with limit set a proper subset of  $\S^{2n-1}$.
Suppose no finite index subgroup of $\Gamma$ stabilizes a totally geodesic submanifold nor a point at infinity.
Then $d_{\Lambda}$ defines a $\Gamma$-invariant metric on the regular set 
$\Omega$. 
\end{corollary}

\begin{proof}
Consider $G$ the connected component of the Zariski closure of $\Gamma$. It is a connected
subgroup of $\PU(n,1)$ which does not stabilizes a totally geodesic subspace. From 
\cite[Thm 4.4.1]{ChenGreenberg}, $G$ contains $\PU(n,1)$. Hence $G$ does not 
stabilize a proper algebraic subset of $\S^{2n-1}$.
Suppose now that some bisector $B$ contains the limit set $\Lambda$ of $\Gamma$. Then $\Gamma$ preserves the
algebraic subset $\bigcap_{\gamma\in \Gamma}\gamma B$ and $\Lambda$ is included in this intersection. 
As it is algebraic, $G$ stabilizes the intersection. This is a contradiction.
\end{proof}

\begin{remark}

It is quite easy to prove that the subgroups $\Gamma$ arising in spherical CR uniformizations of a knot 
or link complement $M$ as in \cite{DerauxFalbel,ParkerWill} does not virtually preserve a 
totally geodesic 
submanifold. Thus to such an uniformization is associated a $\Gamma$-invariant metric on 
the covering $\Omega$ of $M$.  
\end{remark}

\subsection{Infinitesimal metric}

We conclude this section by a reinterpretation of the formula for the infinitesimal metric defined in
section \ref{Finsler metric}.
Observe that the formula for the Finsler metric given in Theorem \ref{theorem:metric} can be written - 
with choices of lifts - for
a tangent vector $v$ at $\omega \in \Omega$ by:
\begin{eqnarray*}
\norm{\omega}{v} = & \max_{\phi, \phi' \in \Lambda} \Re\left(\frac{\bphi'(v)}{\bphi'(\bomega)} - \frac{\bphi(v)}{\bphi(\bomega)}\right)\\
= & \max_{p, p' \in \Lambda} \Re\Bigl\la\dfrac{\bp'}{\la \bp',\bomega \ra}-\dfrac{\bp}{\la \bp,\bomega \ra},v\Bigr\ra
\end{eqnarray*}

In order to have an explicit description of the norm we shall use the Siegel model of two dimensional complex hyperbolic space.
Its boundary is  $2\Re(z_1)+|z_2|^2=0$ (intersected with an affine chart $z_3=1$).  Writing $z_k=x_k+iy_k$, the tangent space in $(z_1,z_2)$ is given by
$$
2dx_1+2x_2dx_2+2y_2dy_2=0.
$$
In particular the tangent space at ${\bf o}=(0,0)$ is given by $dx_1=0$, therefore consists of vectors of the form
$$v_{z,t}=\begin{bmatrix} it\\ z\\ 0 \end{bmatrix}$$
Applying the previous formula in these coordinates we have
$$
\norm{\bf o}{v_{z,t}}=\max_{p, p' \in \Lambda} \Re\Bigl\la\dfrac{\bp'}{\la \bp',\bf o\ra}-\dfrac{\bp}{\la \bp,\bf o\ra},v_{z,t}\Bigr\ra.
$$

\begin{remark}
\begin{enumerate}
 \item The vector $w=\dfrac{\bp'}{\la \bp',\bomega\ra}-\dfrac{\bp}{\la \bp,\bomega\ra}$ is of  positive type.
Indeed, one computes
$$\la w,w\ra= -2\Re\left(\dfrac{\la \bomega,\bp\ra\la \bp,\bp'\ra\la \bp',\bomega\ra}{|\la \bp,\bomega\ra|^2|\la \bp',\bomega\ra|^2}\right)$$
and the triple product $\arg(-\la \bomega,p\ra\la p,p'\ra\la p',\bomega\ra)\in [-\pi/2,\pi/2]$ is Cartan's invariant. 
In particular $w$ is polar to a complex line in $\HdC$.
\item The conditions $\la w,\bp\ra=\la w ,\bp'\ra$ and $\la w, \bomega\ra=0$ imply that $w$ is orthogonal to the complex line defined by $p$ and  
$p'$ passing through $\omega$.
\end{enumerate}
\end{remark}

\section{Complex projective line: $n$-punctured spheres}\label{section:puncturedsphere}

We consider in this section the complex projective line $\CP$. In fact, we restrict
ourselves to a very specific case: let $\Lambda = 
\{p_1,\cdots,p_n\}$ a finite set
 of points in $\C$ and $\Sigma=\CP\setminus\{p_1,\cdots,p_n\}$ its complement, the $n$-punctured sphere.   
 The $n$ linear forms in 
$\Lambda$ are given by $\phi_i=[1 : -p_i]$, so that 
$$\bphi_i \begin{pmatrix}z\\1\end{pmatrix}=z-p_i.$$
We have already observed that when $n=1, 2$ or $3$, the metric is not separating. So
we assume here that $n\geqslant 4$ and $\Lambda$ is not included in a circle. In this
case, $d_\Lambda$ is a metric, as follows from Theorem \ref{thm:metric}. For simplicity,
we denote this metric by $d$.

On the other hand, the surface $\Sigma$ may be equipped with
a hyperbolic metric, denoted by $d_h$. We prove that the infinitesimal metric defined by $d$
 is quasi-isometric to the infinitesimal hyperbolic metric defined by $d_h$.   The reader can guide himself through the proof
 of the following proposition keeping in mind  the physical interpretation of the infinitesimal metric given in Remark \ref{physics}.

\begin{theorem}\label{thm:QI-n-punctured-sphere}
There is a quasi-isometric diffeomorphism between $\bigl(\Sigma,d_h\bigr)$ and $\bigl(\Sigma,d\bigr)$ 
which fixes the punctures.
\end{theorem}
We will show that the Finsler metrics are equivalent up to a diffeomorphism and this will prove the theorem.  
We begin by describing the hyperbolic metric near a cusp point of $\Sigma$.
\begin{lemma}\label{lem:local-hyp}
For  $k=1\cdots n$, there exists a neighbourhood of the puncture $p_k$ on which the infinitesimal hyperbolic metric defined by $d_h$ on $\Sigma$ is 
given by
$$\sqrt{\dfrac{dr^2}{r^2} + r^2dt^2},$$
in the local coordinate $(r,t) \mapsto p_k+re^{i t}$.
\end{lemma}
\begin{proof}
 Note first that the hyperbolic metric on the punctured unit disc is given by 
\begin{equation}\label{eq:hyp-punctured-disc}\dfrac{|dz|}{|z|\log|z|}\end{equation}
This can be seen by pushing forward the hyperbolic metric of the upper half plane, which is $\frac{|dz|}{\Im(z)}$, by the 
holomorphic cover map $z\longmapsto e^{iz}$. Pushing forward \eqref{eq:hyp-punctured-disc} by the diffeomorphism 
given in polar coordinates by $(r,\theta)\longmapsto (-(\log r)^{-1},\theta)$ gives the result.
\end{proof}

\begin{proof}[Proof of Theorem \ref{thm:QI-n-punctured-sphere}]
For any point $m\in\Sigma$, and any tangent vector $v$ at $m$, we denote by
$\|v\|_m  $ the Finsler norm of $v$, and by $\|v\|_m^{h}$ its hyperbolic norm.
We will prove that any diffeomorphism $\varphi: \Sigma \to \Sigma$ which restricts to the
identity around the punctures is a quasi-isometry.  First, consider $\varphi$ such a diffeomorphism, and, 
for each $k=1\cdots n$, let $V_k$ be any neighbourhood of $p_k$ such that $\varphi_{|V_k}$ is the identity.
Since $\Sigma\setminus \underset{k}{\cup} V_k$ is compact, the restriction of $\varphi$ to it is quasi-isometric.
We need therefore only to consider the situation close to the punctures. The result will be proved if we show that for 
each puncture $p_k$, there exists a constant $C>1$ and a neighbourhood $V_k$ of $p_k$ such that for any $m\in V_k$ and 
any tangent vector $v$ at $m$
\begin{equation}
\dfrac{\|v\|_m^h}{C}\leqslant\|v\|_{m}\leqslant C\|v\|_m^{h}.
\end{equation}
We consider the situation close to a fixed puncture, which we assume to be $p_1$. We may chose our coordinates so that 
$p_1=0$.
\begin{figure}[h!]
\begin{center}
\scalebox{0.5}{\includegraphics{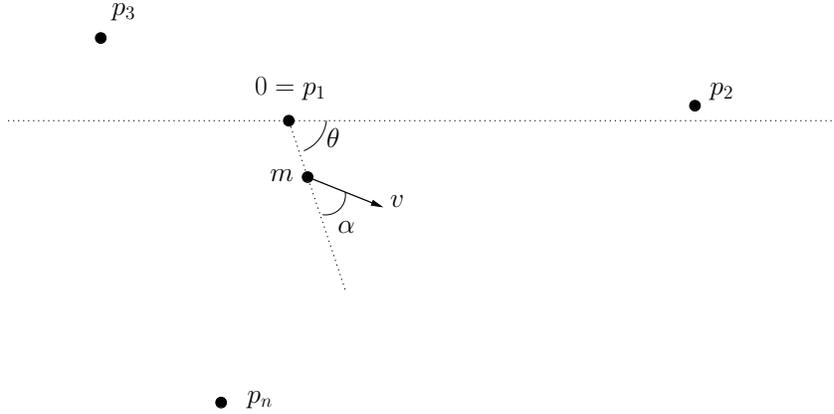}}
\end{center}
\caption{Coordinates \label{fig:coords}}
\end{figure}

In the local coordinate $(r,t) \mapsto re^{it}$, we need to find $C>1$ such that for any tangent vector
$v = a\frac{\partial}{\partial r} + b\frac{\partial}{\partial t}$, we have (see Lemma \ref{lem:local-hyp}) :
\begin{equation}\label{eq:eq}
 \frac{1}{C}\sqrt{\frac{a^2}{r^2} + r^2b^2} \leq \|v\|_m \leq C\sqrt{\frac{a^2}{r^2} + r^2b^2}.
\end{equation}

The rest of the proof is devoted to this local computation. From now on, all constants $C$ that appear are {\it some positive 
constants}, and we remain vague about their values for the sake of readability. The point $m$ and the tangent vector $v$ are 
given by

$$m=\begin{bmatrix}re^{i\theta}\\1\end{bmatrix},v=\begin{bmatrix}\rho e^{i(\theta+\alpha)}\\0\end{bmatrix}.$$

The vector $v$ decomposes as $v =\rho \cos(\alpha) \frac{\partial}{\partial r} + \frac{\rho}{r} \sin{\alpha} \frac{\partial}{\partial t}$. 
Its hyperbolic norm is given by 

$$\|v\|_m^h= \rho \sqrt{\frac{\cos(\alpha)^2}{r^2} + \sin(\alpha)^2}.$$

The points $p_j$ are given by $p_j=k_je^{i\beta_j}$ (with $k_1=0$ and $k_j\neq 0$ for $j\geqslant 2$). By Theorem 
\ref{theorem:metric}, the Finsler norm of $v$ is given by:
\begin{equation} \label{eq:finsler-v}
\|v\|_m=\max_j \Re\Bigl(\dfrac{\phi_j(v)}{\phi_j(m)}\Bigr)-\min_j\Re\Bigl(\dfrac{\phi_j(v)}{\phi_j(m)}\Bigr).
\end{equation}
The quantities involved in \eqref{eq:finsler-v} are given by 
$$\dfrac{\phi_j(v)}{\phi_j(m)}=\dfrac{\rho e^{i(\theta+\alpha)}}{re^{i\theta}-k_je^{i\beta_j}},$$
and by direct computations, we observe (see Remark \ref{physics} for a physical interpretation of the following formula).
\begin{itemize}
\item If $j=1$ ( and so $k_1=0$ ),
\begin{equation}\label{eq:expansion1}\dfrac{1}{\rho}\Re\Bigl(\dfrac{\phi_1(v)}{\phi_1(m)}\Bigr)=\dfrac{\cos(\alpha)}{r}.
\end{equation}
\item If $j\geqslant 2$ (and thus $k_j\neq 0$),
\begin{equation}\label{eq:expansionj}\dfrac{1}{\rho}\Re\Bigl(\dfrac{\phi_j(v)}{\phi_j(m)}\Bigr)=-\dfrac{\cos(\theta+\alpha-\beta_j)}{k_j} + O(r)\end{equation}
\end{itemize}

From physical considerations one expects that the most important term in the contribution to the metric near $p_1$ is the one corresponding to $j=1$ if $\alpha\neq \pm \pi/2$ and from another point which is not aligned with $m$ and $p_1$
if $\alpha=\pm \pi/2$.  In the  following  we establish the estimates confirming this intuition.

Let us consider the right-hand side inequality in \eqref{eq:eq}. First, we note that for all $\gamma\in\R$, 
$\cos(\alpha+\gamma)\leqslant |\cos\alpha| +|\sin\alpha|$. Therefore,

$$\underset{j\geqslant 2}{\max}\left(\dfrac{-\cos(\theta+\alpha-\beta_j)}{k_j}\right)\leqslant
\dfrac{|\cos\alpha| +|\sin\alpha|}{\underset{j\geqslant 2}{\min}\, k_j}.$$

Moreover, for any finite  subset $F$ of $\R$, we have $\max(F)-\min(F)\leqslant 2\max|F|$. Provided that $r$ is small 
enough, this implies 
\begin{eqnarray}
 \|v\|_m^h & \leqslant & 2\underset{j}{\max}\left|\Re\dfrac{\phi_j(v)}{\phi_j(m)}\right| \nonumber\\
  & \leqslant & 2\rho\left(\dfrac{|\cos\alpha|}{r}+C(|\cos\alpha|+|\sin\alpha|)\right)\nonumber \\
 & = & \rho\left(\dfrac{C}{r}|\cos\alpha| + C|\sin\alpha|\right)\nonumber\\
& \leqslant & C \rho\left(\dfrac{|\cos\alpha|}{r}+|\sin\alpha|\right) \nonumber\\
&\leqslant & C\rho \sqrt{\dfrac{\cos\alpha^2}{r^2}+\sin^2\alpha} 
\end{eqnarray}

We know consider the left-hand side inequality in \eqref{eq:eq}. By symetry, we may restrict ourselves to the case where 
$\alpha\in[-\pi/2,\pi/2]$. The geometric idea is the following. When $\alpha$ is close to $0$, the $\max$ in \eqref{eq:eq} is reached for $i=1$ 
if $r$ is small enough, due to the $\cos (\alpha)/r$ term. When $|\alpha|$ becomes larger, the $j=1$ term becomes 
less influent, and the Finsler norm of $v$ is obtained from those terms where $j>1$. In the extreme case where 
$\alpha=\pm\pi/2$, the $j=1$ term vanishes.

To make this idea more precise, we first compute for any $j\leqslant 2$
\begin{eqnarray}
\dfrac{1}{\rho}\Bigl(\Re(\dfrac{\phi_1(v)}{\phi_1(m)})-\Re(\dfrac{\phi_j(v)}{\phi_j(m)}) \Bigr)
& = & \dfrac{\cos(\alpha)}{r}-\dfrac{r\cos\alpha -k_j\cos{(\alpha+\theta-\beta_j)}}{|re^{i\theta}-k_je^{i\beta_j}|^2}\nonumber\\
& = & \dfrac{k_j\left(k_j\cos\alpha-r\sin{(\theta-\beta_j)}\sin\alpha\right)}{r|re^{
i\theta}-k_je^{i\beta_j}|^2} \nonumber
\end{eqnarray}
We observe that the latter quantity has the same sign as $k_j\cos\alpha-r\sin{(\theta-\beta_j)}\sin\alpha$. A direct 
resolution shows that it is non-negative if and only if $\alpha$ belongs to an interval $I_j$, which is defined by
$I_j=[-\pi/2,\alpha_j]$ if $\sin(\theta-\beta_j) >0$, by $I_j=[\alpha_j,\pi/2]$ if $\sin(\theta-\beta_j) <0$, where 
$\alpha_j$ is determined by 
$$\tan(\alpha_j)=\dfrac{k_j-r\cos(\theta-\beta_j)}{r\sin{(\theta-\beta_j)}},$$
and $I_j=-[\pi/2,\pi/2]$ if $\sin(\theta-\beta_j) =0$.

As a consequence, the set of values of $\alpha$ for which the max is reached for $j=1$ is a subinterval $I$ of 
$[\pi/2,\pi/2]$ which contains $0$ in its interior.

\begin{itemize}

\item For $\alpha\in I$, we have thus

$$\|v\|_m \geqslant \rho \left(\dfrac{\cos\alpha}{r}-\max\dfrac{2}{k_j}\right)\geqslant\|v\|_m^h.$$

\item For $\alpha$ outside $I$, we have

$$\|v\|_m\geqslant\rho \underset{j,\ell\geqslant 2}{\max}\left(\Re\dfrac{\phi_j(v)}{\phi_j(m)}-\Re\dfrac{\phi_\ell(v)}{\phi_\ell(m)}\right).$$

In view of \eqref{eq:expansionj}, this implies (provided that $r$ is small enough)

$$\|v\|_m\geqslant\rho \underset{j,\ell\geqslant 2}{\max}\Bigl| \dfrac{\cos(\theta+\alpha-\beta_j)}{k_j}-\dfrac{\cos(\theta+\alpha-\beta_\ell)}{k_\ell}\Bigr|-C\rho r.$$

Now, the max in the right hand side of the previous inequality is not zero since the points $p_j$ do not lie on a circle. 
Therefore, making $r$ smaller if necessary, we have 

$$\|v\|_m\geqslant \rho C \geqslant \rho|\sin\alpha|\geqslant C \|v\|_m^h.$$

\end{itemize}
This concludes the proof.
\end{proof}

\begin{remark} \label{physics} One can gain intuition about this metric with a physical analogy. Indeed, the contribution 
of each point $p_i$ in the definition of the Finsler metric $\|v\|_m$ corresponds to a magnetic field ${\bf B}_i$ induced 
by a constant current passing through an infinite line perpendicular to the plane $\C$ at the point $p_i$ 
(see equation \ref{eq:expansion1}).  The magnetic field is tangent to the circles centred at $p_i$ and decreases with the 
inverse of the distance. The force ${\bf F}_i$ on a charged particle at $m$ with velocity $v$ moving on the magnetic field 
${\bf B}_i$ is given by  ${\bf F}_i= v\wedge {\bf B}_i$ and it can be considered a scalar (the vector ${\bf F}_i$ is 
perpendicular to the plane). The Finsler metric is then
$$
\|v\|_m=\max_i{\bf F}_i-\min_j{\bf F}_j.
$$
This suggests other natural infinitesimal metrics as various combinations or means of magnetic forces but the 
previous definition has the advantage of making it clear that  this metric is always  strictly positive if the 
points are not on the same circle or line.
\end{remark}

\section{Real projective line}\label{section:RP1}

In this section, we consider the case where $\Lambda$ is a compact subset of 
$\RP$ and $\Omega$ its complement. The open set $\Omega$ is a union of pairwise disjoint intervals, and 
the distance $d_\Lambda$ can be computed for points in different components of $\Omega$. In that case, $d_\Lambda$ 
has a close connection with the hyperbolic distance on the disc $\partial \Delta$, which is hardly a surprise since 
this distance is induced by the cross-ratio on $\partial\Delta$.



\subsection{Distance between intervals}

Let us consider a closed set $\Lambda \subset \RP$ and its complement $\Omega$. As above, we identify 
$\Lambda$ with a subset of the dual $\R P^1$: to any point $\lambda \in \RP$ corresponds the linear map 
$(x_1,x_2)\longmapsto x_1-\lambda x_2$, which corresponds to the affine map $x\longmapsto x-\lambda$ in 
the affine chart $\{x_2=1\}$. Under this identification, we will sometimes refer to points in $\Lambda$ as forms, 
and, for instance call them $\phi$, $\phi'$... The open subset $\Omega $ is a union of open intervals, its connected 
components, to which we will refer as its \emph{components}. The metric $d_\Lambda$ is invariant under the subgroup 
of $\PGL(2,\R)$ preserving $\Lambda$ and is particularly interesting when this group is large as, for example, when 
$\Lambda$ is the limit set of a Fuchsian group. 

Viewing $\RP$ as the boundary at infinty of the Poincar\'e disc $\Delta$, there is a close relation between the cross-ratio 
on $\RP$ and the hyperbolic metric in $\Delta$ : 

\begin{lemma}\label{lem:cross-hyp}
If $a$, $b$, $c$ and $d$ are pairwise distinct points in $\RP$, the cross ratio satisfies $|[a,b,c,d]|=e^\delta$, where 
$\delta$ is the (hyperbolic) distance between the orthogonal projections of $c$ and $d$ onto the geodesic spanned by $a$ 
and $b$. 
\end{lemma}

\begin{proof}
 Applying an element of PGL(2,$\R$), which preserves the cross-ratio, we may assume that $a=\infty$, $b=0$, $c=1$ and $d=x>0$.
Then the projections of $c$ and $d$ onto the geodesic $(ab)$ are $i$ and $ix$, which are a distance $\log (x)$ apart. The cross 
ratio $[a,b,c,d]$ is equal to $x$.
\end{proof}

In order to understand the metric $d_\Lambda$ in this case, we first start with the following lemma, 
which shows that the metric between points in $\Omega$ depends only on the (at most) two components 
containing the points.  

\begin{lemma} \label{lem:dist-RP1}
\begin{enumerate}
 \item The metric $d_\Lambda$ restricts to the Hilbert metric on each component of $\Omega$.
 \item If $\omega$ and $\omega'$ belong to distinct components $I$ and $I'$ of $\Omega$, then the max 
defining the distance $d_\Lambda(\omega,\omega')$ is realised for $\phi$ and $\phi'$ associated to boundary 
points of $I'$ and $I$ respectively.
\end{enumerate}
\end{lemma}
 
\begin{proof}
These facts are very natural geometrically in view of Lemma \ref{lem:cross-hyp} (see Figure \ref{fig:proj}).  
Choose an affine chart such that the two points in $\Omega$ are $\omega = 0$ and $\omega' = \infty$. 
Then we have:
\begin{eqnarray}
d_\Lambda (\omega,\omega') &=& \underset{\lambda, \lambda'\in\Lambda} \max \log 
\left|\dfrac{(\omega-\lambda)(\omega'-\lambda')}{(\omega-\lambda')(\omega'-\lambda)}\right|\nonumber\\
& = & \underset{\lambda, \lambda'\in\Lambda}\max{\log}\left|\dfrac{\lambda}{\lambda'}\right|\label{log}
\end{eqnarray}
The component $I$ of $\Omega$ containing $\omega = 0$ is a neighborhood of $0$ and contains no point of $\Lambda$. 
Similarly, the component $I'$ containing $\omega' = \infty$ is a neighborhood of $\infty$ that 
contains no point of $\Lambda$. In order to maximize \eqref{log}, one 
should choose $\lambda$ as big as possible in absolute value, 
i.e. one of the endpoints of $I'$ and $\lambda'$ as small as possible, i.e. one of the 
endpoints of $I$. This proves the second item. In case $I$ and $I'$ are equal, then the max is attained 
when $\lambda$ and $\lambda'$ are distinct boundary points. This means that the distance between $\omega$ 
and $\omega'$ is just their Hilbert distance. \end{proof}

\begin{figure}[h!]
\begin{center}
\scalebox{0.5}{\includegraphics{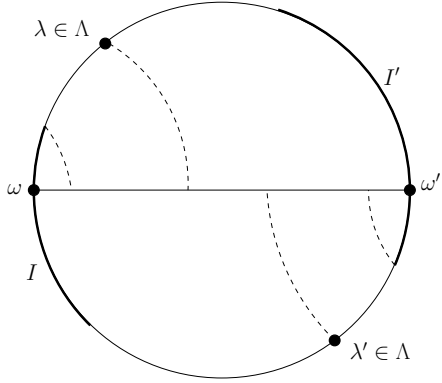}}
\end{center}
\caption{Lemma \ref{lem:dist-RP1} : $I$ and $I'$ are components of $\Omega$, points of $\Lambda$ that are not 
endpoints of $I$ and $I'$ give projections onto the geodesic $(\omega,\omega')$ that are closer than those of the 
endpoints of $I$ and $I'$. \label{fig:proj}}
\end{figure}

Given two components $I$ and $I'$ of which closures are disjoint, the infimum of $d_\Lambda(x,x')$ 
over $x\in I$ and $x'\in I'$ is not $0$: as $\Omega$ is the complement of $\Lambda$, the distance $d_\Lambda$ 
on $\Omega$ is proper and the infimum is attained. Hence we define the distance $d_\Lambda(I,I')$ between 
$I$ and $I'$  to be this infimum. The following lemma gives a beautiful geometrical interpretation of this distance, 
which is not surprise in view of Lemma \ref{lem:cross-hyp}.

\begin{lemma}
The distance between two components $I$ and $I'$ in $\Omega$ is given by the distance 
between the two geodesics $\gamma$ and $\gamma'$ in the hyperbolic space where 
the endpoints of $\gamma$ are the endpoints of $I$ and those of $\gamma'$ are the endpoints of $I'$.
\end{lemma}
\begin{proof}
One can arrange, up to the action of $\PGL(2,\R)$, the two intervals to be 
$]-1,1[$ and the complement of $[-a,a]$ (with $1<a$). In this case, the hyperbolic distance 
between the two geodesics is $\log(a)$ (it is the hyperbolic distance between $i$ and $ia$).
One can then compute that the minimal distance between points in both intervals is attained for 
$0$ and $\infty$ and is precisely $d(0,\infty)=\ln [\infty,0,1,a]=\log a$. Again, this can be easily seen more 
geometrically (see Figure \ref{fig:dist-intervals}).
\end{proof}

\begin{figure}[h!]
\begin{center}
\scalebox{0.5}{\includegraphics{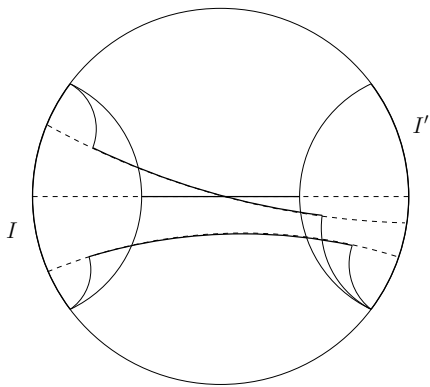}}
\end{center}
\caption{Lemma \ref{lem:dist-RP1} : the distance between $I$ and $I'$ is realised at the common orthgonal of 
the associated geodesics. \label{fig:dist-intervals}}
\end{figure}

Remark that $[a,-a,-1,1]=\left (\frac{1+a}{1-a}\right)^2$. Therefore, the distance $d$ between the two intervals 
$I=]-a,a[$  and the complement of $I'=\RP\setminus[-1,1]$ is related to the cross ratio of the four endpoints by 
the relation
$
[a,-a,-1,1]=\coth^{2}{\frac{d}{2}}.
$
Reading backwards and using the invariance of the distance by projective transformation, we get that 
the distance $d_\Lambda(I,I')$ between two intervals $I = ]a,b[$ and $I' = ]x,y[$ 
(with $a<b<x<y$) is then
$$
d_\Lambda(I,I') = 2\tanh(\sqrt{[a,b,x,y]}).
$$

This distance leads to the definition of a measure on $\RP$ associated to a closed set $\Lambda \subset  \RP$ 
and a chosen component of $\RP\setminus \Lambda$.

\begin{definition}
 Let $\Lambda\subset \RP$ be a closed set and $\Omega$ its complement. We denote by $\mathcal{A}_\Lambda$ the 
sigma algebra generated by components of $\Omega$ and Borel sets in $\Lambda$.
\end{definition}

In other words , a measurable set $E$ is an union 
$$\left(\underset{I'\textrm{ component of }\Omega} \bigcup I'\right) \cup B
\mbox{, where $B$ is a Borel subset of $\Lambda$.}$$

For any component $I$ of $\Omega$, we denote by $\gamma_I$ the geodesic in $H^2_\R$ whose 
endpoints are the same as those of $I$. Then the hyperbolic length induces a measure on $\Lambda$, 
obtained by pulling back the hyperbolic length element along $\gamma_I$ by the orthogonal projection.
We denote this measure by $\nu_I$.

\begin{definition}\label{def:measure}
The measure on $\mathcal{A}_\Lambda$ associated with $I$, denoted by $\mu_{\Lambda,I}$ is defined by
\begin{itemize}
 \item For any component $I'$ of $\Omega$, $\mu_{\Lambda,I}(I')=2\ln (\coth \frac{d}{2})$ where $d$ is 
the distance between $I$ and $I'$.
 \item For any Borel set $B$ in $\Lambda$ $\mu_{\Lambda,I}(B)=\nu_I(B)$.
\end{itemize}
\end{definition}

As an example, if the component $I$ is $(-\infty,0)$, then $\nu_I$ is the measure on Borel sets of $\R_+$ given by 
$d\nu_I=\frac{dx}{x}$. In that case the geodesic $\gamma_I$ is the one connecting $0$ to $\infty$ in the upper 
half-plane. In $E\in\mathcal{A}_\Lambda$ is given by $E=\left(\cup I'\right) \cup B$, then
 
$$\mu_{\Lambda,I}(E)= 2 \sum_{I'} \ln \left(\coth\left(\frac{d_\Lambda(I,I')}{2}\right)\right) + \int_B \frac{dx}{x}. $$



\subsection{Self-similar closed sets in $\RP$ and Basmajian formula}
In this section, we revisit the famous Basmajian formula for surfaces with boundary from the point of view 
of  the metric $d_\Lambda$, in the case where $\Lambda$ is a self-similar closed set. 
We refer the reader to Basmajian's work \cite{Basmajian}, or to the expository articles \cite{Calegari,McShane}, for more 
details. The general form of the Basmajian formula relates the orthospectrum  of a hyperbolic manifold with geodesic boundary 
to the area of the geodesic boundary. We restrict here to the case of surfaces.  We propose here a slight generalization for 
$\Lambda$ not the full limit set of a Fuchsian group, but just some closed set with self-similarity properties.

\begin{definition}
A  subset  $\Lambda\subset \RP$ is \emph{projectively self-similar} if 
there exists a finite family of projective maps $\{f_s\}$ such that
$\Lambda=\bigcup_s f_s(\Lambda)$. The maps $f_s$ are called 
\emph{self-similarities} of $\Lambda$.
\end{definition}

One has to keep in mind  the fundamental example of the limit set of a Fuchsian group.   
In that case, the family can be reduced to a unique map. Another example is given by the 
usual triadic Cantor set in the interval $[0,1]$ to which we add the point $\infty$. In 
this case, the set of self-similarities are the contractions of 
ratio $\frac 13$ and center $0$ or $1$. Both these transformations fix $\infty$.

Set, as before $\Omega = \RP \setminus\Lambda$ and suppose that there exists a 
component $I = (a,b)$ of $\Omega$ which is preserved under a self-similarity $f$. The map $f$ is a 
hyperbolic element fixing $a$ and $b$. Choose an interval $D = [x,f(x)]\subset  \RP \setminus I$ such 
that $x \in \Lambda$. It is a fundamental domain for the action of $f$ on $\RP \setminus I$.  
The point $f(x)$ is also in $\Lambda$ by invariance. In the example of the triadic Cantor set, 
the interval $I$ is $(\infty,0)$ for the contraction about $0$ and $(1,\infty)$ for the second 
contraction.

The self-similarities act on the set of components of $\Omega$, and in turn, act on the set of lengthes between 
components. We call {\it orthospectrum} of $\Lambda$ the set of distances between components of $\omega$ modulo 
the action of self-similarities. The following result relates the ratio of contraction of the self-similarity, 
or rather its translation length in the hyperbolic space, to the measure $\mu_{\Lambda,I}$ of $D$, recovering the 
a version of the Basmajian formula.

\begin{theorem}\label{theorem:basmajian} Let $\Lambda\subset  \RP$ be  closed and preserved under a 
hyperbolic element $f\in \PGL(2,\R)$ with translation length $l$, let $\Omega$ be its complement. 
Suppose, furthermore, that $f$ preserves a component  $I_f\subset\Omega$.
Let $D$ be a fundamental domain as above. Then, we have:
$$
l=\mu_{\Lambda_f}(D).
$$
where $\mu_{\Lambda_f}$ is the canonical measure defined by $I_f$.
\end{theorem}

Remark that the formula above can be written as

$$
l=S'+\int_{I_\Lambda} d\nu_I 
$$

with 

$$S'= \sum 2\ln \left(\coth \frac{d(I,I')}{2}\right),$$ 

where the sum ranges over components $I'$ of $\Omega$ inside $D$. In the case of limit 
sets of Fuchsian groups the continuous measure does not appear as the limit set always 
has measure zero and the above formula reduces to the Basmajian's identity for hyperbolic 
surfaces. Observe also that there exists one such identity for each hyperbolic element 
preserving a component of the complement $\RP\setminus \Lambda$. Note 
moreover that if a self-similarity with fixed points $a\neq b$ does not fix
a component, we may split $\Omega$ and $\Lambda$ in two according to the 
sides of $(a,b)$. We retrieve then two formulas for $l$.

\begin{proof}
By projective invariance, we may assume that $I_f$ is $(-\infty, 0)$, so the axis of $f$ is 
the vertical geodesic above $0$ in the upper half-plane.
Let $D_\Lambda=D\cap \Lambda$ and $D'_\Lambda=D\setminus D_\Lambda = \Omega \cap D$.  
For each component $I_i$ of $D'$, one computes its distance $d_i$ to $I_f=(-\infty,0)$.  
This is  the \emph{orthospectrum} of $\Lambda$ with respect to the 
interval $(-\infty,0)$ along the fundamental domain $D$. 
Each distance $d_i$ is related to a quadrilateral in the hyperbolic  half plane.  
The proof then is then a simple observation guided by Figure \ref{fig:basmajian}.  
Indeed, the translation length $l$ of the hyperbolic element $f$ 
is the integral of the hyperbolic Lebesgue measure $d\nu$ over a fundamental interval for 
the action of $f$ by translation along its axis. In the present case, it is the 
vertical geodesic in the hyperbolic upper half plane, and so. 
As in the example following Definition \ref{def:measure}, the integral is computed
 as a sum of two terms:  one term corresponds to 
the integration of $dx/x$ along the closed set $D_\Lambda = D\cap \Lambda$ and 
the other corresponds to the integration over the components in $D' = D\cap \Omega$. For each such 
component $I_i = (x,y)$ (with $x<y$)the integral of $dx/x$ on $I_i$ is computable and relates to the 
distance $d_i := d_\Lambda(I,I_i)$ by (compare \cite{Basmajian}):
\begin{eqnarray*}
\int_{I_i} dx/x   = & \ln \dfrac yx\\
 = & \log [\infty, 0, x,y]\\
 = & 2\log \left(\coth \left(\frac{d_i}{2}\right)\right)\\
 = & \mu_{\Lambda,I}(I_i)
\end{eqnarray*}
This proves the result.

\begin{figure}[h!]
\begin{center}
\scalebox{0.5}{\includegraphics{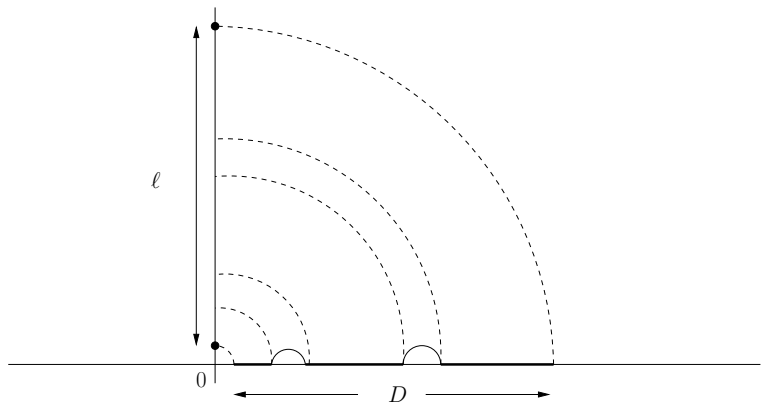}}
\end{center}
\caption{Basmajian's formula. The marked points on the vertical axis bound a fundamental 
interval for the action of $f$ on its axis. The translation length of $f$ is $\ell$. \label{fig:basmajian}}
\end{figure}

      
       


\end{proof}

\subsection{Quasi-M\"obius maps and quasi-isometry}

Originating from \cite{CooperPignataro}, there is a literature
about self-similar Cantor sets up to bi-Lipschitz transformations.
In the case that is relevant to us, Cooper and Pignataro \cite{CooperPignataro}
classify the self-similar Cantor subsets of $[0,1]$ up to order
preserving bi-Lipschitz maps (or in another language, quasi-isometries).

We remark here that if two closed sets $\Lambda$ and $\Lambda'$ are 
quasi-M\"obius -- 
the right notion extending quasi-isometry as we will see --
then their complement $(\Omega,d_\Lambda)$ and $(\Omega',d_\Lambda')$ 
are quasi-isometric.

Let us begin by recalling a few definitions. We fix a cyclic orientation
on the real projective line $\RP$ and every affine coordinates we will
consider will respect this ordering. 
\begin{definition}[\cite{Vaisala}]
An order preserving invertible map $F$ between to subsets $E$ and $E'$ of 
$\RP$ is \emph{quasi-M\"obius} if there is a constant $K\geq 1$ such that for any $4$-tuples 
$(e_0,e_1,e_2,e_3)$ of distinct elements of $E$, we have:
$$\frac 1K \leq \frac{[F(e_0),F(e_1),F(e_2),F(e_3)]}{[e_0,e_1,e_2,e_3]}\leq K.$$
\end{definition}
Note that this class of maps has close ties with the quasi-symmetric maps
\cite{Vaisala}, that we will not use.

Suppose you have an order-preserving map $f$ between two Cantor 
subsets of $[0,1]$ -- in fact, any 
two compact subsets of $\R$.
Define $\Lambda = C\cup\{\infty\}$ and $\Lambda' = C'\cup\{\infty\}$. 
Let $F$ be the extension
of $f$  to the map between $\Lambda$ and $\Lambda'$ which fixes $\infty$. 
Then it is easy to check 
that $f$ is bi-Lipschitz if and only if $F$ is
quasi-M\"obius. The following theorem explains that such a 
quasi-M\"obius map extends to a quasi-M\"obius
map of $\RP$. It implies in turn that the complements $\Omega,d_\Lambda$ 
and $\Omega',d_{\Lambda'}$ are
quasi-isometric.
\begin{theorem}\label{RP1:qi}
Let $\Lambda$ and $\Lambda'$ be two closed subset of $\RP$ and denote 
by $\Omega$ and $\Omega'$ their complements.

Then any order-preserving quasi-M\"obius invertible map $F : \Lambda \to \Lambda'$ extends 
to a quasi-M\"obius invertible map $\bar F: \RP \to \RP$. Moreover, the restriction of $\bar F$ to 
$\Omega$ realizes a quasi-isometry between $(\Omega,d_\Lambda)$ and $(\Omega',d_{\Lambda'})$.
\end{theorem}  

\begin{proof}
We normalize $\Lambda$, $\Lambda'$ and $F$ in the following way: we suppose,
up to a M\"obius transformation that $\Lambda$ and $\Lambda'$ contain 
$0$, $1$ and $\infty$ and that $F$ fixes $0$, $1$,
$\infty$. Note that, as the order is preserved, two points $x,x'$ 
in $\Lambda$ are the endpoints of 
a component $(x,x')$ of $\Omega$ if and only if $(F(x),F(x'))$ 
is a component of $\Omega'$. Using the almost invariance of 
cross-ratios of the form $[\infty,0,1,t]$ 
for $t\neq 0\in \Lambda$, we get that $$\frac 1K\leq \frac{F(t)}{t}\leq K.$$ 
Using now cross-ratios $[\infty,t,0,t'] = \frac{t-t'}{t}$ for 
$t,t'\neq 0\in \Lambda$, we get that:
$$\frac 1K\leq \frac{F(t)-F(t')}{t-t'}\frac{t}{F(t)}\leq K.$$ 
The first inequality grants that $F$ is $K^2$-bi-Lipschitz in 
restriction to $\R$.

Define now the extension $\bar F$ in the following way: $\bar F = F$ on $\Lambda$. On each bounded 
component $(x,x')$,
$\bar F$ is the unique affine bijection between $(x,x')$ and $(F(x),F(x'))$. 
On an unbounded component
$(x,\infty)$ or $(\infty,x')$, then $\bar F$ is the unique translation
bijection between this component and its image. Note that $\bar F_{|\R}$ 
is an affine extension of a $K^2$-bi-Lipschitz map: it is itself 
$K^2$-bi-Lipschitz. We claim that $\bar F$ is a quasi-M\"obius map
with constant $K^8$. 

Consider indeed $4$ distinct points $(x_0,x_1,x_2,x_3)$ in $\RP$. 
Up to transformations of the cross-ratio, 
we assume that $x_0<x_1<x_2<x_3<\infty$. If $x_0 = \infty$, we have:
$$[\infty, \bar F(x_1),\bar F(x_2),\bar F(x_3)] = \frac{\bar F(x_1) - \bar F(x_3)}{\bar F(x_1) - \bar F(x_2)} \leq K^4 \frac{x_1-x_3}{x_1-x_2} = K^4 [\infty,x_1,x_2,x_3].$$
The inequality is obtained using the $K^2$-bi-Lipschiptz property of $\bar F$. One shows similarly the
minoration leading to:
$$\frac{1}{K^4} \leq \frac{[\infty, \bar F(x_1),\bar F(x_2),\bar F(x_3)]}{[\infty,x_1,x_2,x_3]} 
\leq K^4.$$
If all four points are real, a very similar computation leads to 
$$\frac{1}{K^8} \leq \frac{[\bar F(x_0), \bar F(x_1),\bar F(x_2),\bar F(x_3)]}{[x_0,x_1,x_2,x_3]} 
\leq K^8.$$
So $\bar F$ is $K^8$-quasi-M\"obius. This proves the first claim of 
the theorem. The second claim easily follows. Indeed, for each $p$, 
$p'$ in $\Lambda$, $\omega$, $\omega'$ in $\Omega$, we compute:
\begin{eqnarray*}
d_{\Lambda'} \left(\bar F(\omega),\bar F(\omega')\right) 
= & \underset{q,q' \in \Lambda'}\max \ln|[q,q',\bar F(\omega),\bar F(\omega')]|\\
= &\underset{p,p' \in \Lambda}\max \ln|[\bar F(p),\bar F(p'),\bar F(\omega),\bar F(\omega')]|\\
\leq & \underset{p,p' \in \Lambda}\max \ln|K^8[p,p',\omega,\omega']|\\
\leq & d_\Lambda(\omega,\omega') + 8\ln K
\end{eqnarray*}
Conversely, a similar computation shows that:
$$d_\Lambda(\omega,\omega') - 8\ln K\leq d_{\Lambda'} \left(\bar F(\omega),\bar F(\omega')\right)
\leq d_\Lambda(\omega,\omega') + 8\ln K.$$
This proves the theorem.
\end{proof}

\bibliographystyle{plain}
\bibliography{biblio}

\begin{thebibliography}{10}

\bibitem{Basmajian}
Ara Basmajian.
\newblock The orthogonal spectrum of a hyperbolic manifold.
\newblock {\em Amer. J. Math.}, 115(5):1139--1159, 1993.

\bibitem{Benoist}
Yves Benoist.
\newblock Convexes divisibles.
\newblock {\em C. R. Acad. Sci. Paris S\'er. I Math.}, 332(5):387--390, 2001.

\bibitem{Benoist-survey}
Yves Benoist.
\newblock A survey on divisible convex sets.
\newblock In {\em Geometry, analysis and topology of discrete groups}, volume~6
  of {\em Adv. Lect. Math. (ALM)}, pages 1--18. Int. Press, Somerville, MA,
  2008.

\bibitem{BS}
D.~Jr Burns and S.~Schnider.
\newblock Spherical hypersurfaces in complex manifolds.
\newblock {\em Inv. Maths}, 33:223--246, 1976.

\bibitem{Calegari}
Danny Calegari.
\newblock Chimneys, leopard spots and the identities of {B}asmajian and
  {B}ridgeman.
\newblock {\em Algebr. Geom. Topol.}, 10(3):1857--1863, 2010.

\bibitem{CanoLiuLopez}
Angel Cano, Bingyuan Liu, and Marlon~M. L\'opez.
\newblock The limit set for discrete complex hyperbolic groups.
\newblock {\em Indiana Univ. Math. J.}, 66(3):927--948, 2017.

\bibitem{CanoParkerSeade}
Angel Cano, John~R. Parker, and Jos\'e Seade.
\newblock Action of {$\mathbb R$}-{F}uchsian groups on {$\mathbb{CP}^2$}.
\newblock {\em Asian J. Math.}, 20(3):449--473, 2016.

\bibitem{C}
E.~Cartan.
\newblock Sur la g\'eom\'etrie pseudo-conforme des hypersurfaces de deux
  variables complexes.
\newblock {\em I. Ann. Math. Pura. Appl.}, 11:17--90, 1932.

\bibitem{ChenGreenberg}
S.~S. Chen and L.~Greenberg.
\newblock Hyperbolic spaces.
\newblock In {\em Contributions to analysis (a collection of papers dedicated
  to {L}ipman {B}ers)}, pages 49--87. Academic Press, New York, 1974.

\bibitem{CooperPignataro}
Daryl Cooper and Thea Pignataro.
\newblock On the shape of {C}antor sets.
\newblock {\em J. Differential Geom.}, 28(2):203--221, 1988.

\bibitem{delaHarpe}
Pierre de~la Harpe.
\newblock On {H}ilbert's metric for simplices.
\newblock In {\em Geometric group theory, {V}ol.\ 1 ({S}ussex, 1991)}, volume
  181 of {\em London Math. Soc. Lecture Note Ser.}, pages 97--119. Cambridge
  Univ. Press, Cambridge, 1993.

\bibitem{DerauxFalbel}
Martin Deraux and Elisha Falbel.
\newblock Complex hyperbolic geometry of the figure-eight knot.
\newblock {\em Geom. Topol.}, 19(1):237--293, 2015.

\bibitem{Guilloux-padic}
Antonin Guilloux.
\newblock Yet another {$p$}-adic hyperbolic disc: {H}ilbert distance for
  {$p$}-adic fields.
\newblock {\em Groups Geom. Dyn.}, 10(1):9--43, 2016.

\bibitem{McShane}
Greg McShane.
\newblock Geometric identities.
\newblock {\em RIMS Kokyuroku}, 1836:94--103, 2013.

\bibitem{PapadopoulosTroyanov}
Athanase Papadopoulos and Marc Troyanov.
\newblock From {F}unk to {H}ilbert geometry.
\newblock In {\em Handbook of {H}ilbert geometry}, volume~22 of {\em IRMA Lect.
  Math. Theor. Phys.}, pages 33--67. Eur. Math. Soc., Z\"urich, 2014.

\bibitem{HandbookHilbert}
Athanase Papadopoulos and Marc Troyanov, editors.
\newblock {\em Handbook of {H}ilbert geometry}, volume~22 of {\em IRMA Lectures
  in Mathematics and Theoretical Physics}.
\newblock European Mathematical Society (EMS), Z\"urich, 2014.

\bibitem{PapadopoulosTroyanov-weakMinkowski}
Athanase Papadopoulos and Marc Troyanov.
\newblock Weak {M}inkowski spaces.
\newblock In {\em Handbook of {H}ilbert geometry}, volume~22 of {\em IRMA Lect.
  Math. Theor. Phys.}, pages 11--32. Eur. Math. Soc., Z\"urich, 2014.

\bibitem{ParkerWill}
John~R. Parker and Pierre Will.
\newblock A complex hyperbolic {R}iley slice.
\newblock {\em Geom. Topol.}, 21(6):3391--3451, 2017.

\bibitem{Troyanov}
Marc Troyanov.
\newblock Funk and {H}ilbert geometries from the {F}inslerian viewpoint.
\newblock In {\em Handbook of {H}ilbert geometry}, volume~22 of {\em IRMA Lect.
  Math. Theor. Phys.}, pages 69--110. Eur. Math. Soc., Z\"urich, 2014.

\bibitem{Vaisala}
Jussi V\"ais\"al\"a.
\newblock Quasi-{M}\"obius maps.
\newblock {\em J. Analyse Math.}, 44:218--234, 1984/85.

\bibitem{wolf}
Joseph~A. Wolf.
\newblock Fine structure of {H}ermitian symmetric spaces.
\newblock pages 271--357. Pure and App. Math., Vol. 8, 1972.

\end{thebibliography}

\begin{flushleft}
  \textsc{E. Falbel, A. Guilloux\\
  Institut de Math\'ematiques de Jussieu-Paris Rive Gauche \\
CNRS UMR 7586 and INRIA EPI-OURAGAN \\
 Sorbonne Universit\'e, Facult\'e des Sciences \\
4, place Jussieu 75252 Paris Cedex 05, France \\}
 \verb|elisha.falbel@imj-prg.fr, antonin.guilloux@imj-prg.fr|
 \end{flushleft}
\begin{flushleft}
  \textsc{P. Will\\
  Institut Fourier, Universit\'e de Grenoble 1, BP 74, Saint Martin d'H\`eres Cedex, France}\\
  \verb|will@univ-grenoble-alpes.fr|
\end{flushleft}

\end{document}